\documentclass[a4paper,11 pt]{amsart}

 \usepackage[sc]{mathpazo}
\usepackage{eulervm}
\usepackage{hyperref}

\usepackage[utf8]{inputenc}
\usepackage[T1]{fontenc}

\usepackage[english]{}
\usepackage{mathtools}
\usepackage{braket}
\usepackage{amsmath, amssymb, stmaryrd, mathabx}
\usepackage{enumerate}
\usepackage{tikz-cd}
\usepackage{tikz,float, hyperref}
\usepackage{titlesec}

\titleformat{\section}
{\normalfont\normalsize\bfseries}{\thesection.}{1em}{}
\titleformat{\subsection}
{\normalfont\normalsize\itshape}{\thesubsection.}{1em}{}
\titleformat{\subsubsection}
{\normalfont\normalsize\itshape}{\thesubsubsection.}{1em}{}

\newcommand{\scp}[2]{\langle #1, #2\rangle}

\newcommand{\bta}[1]{\beta^{[#1]}}

\newcommand{\Z}{\mathbb{Z}}
\newcommand{\R}{\mathbb{R}}
\newcommand{\C}{\mathbb{C}}

\newcommand{\res}{\operatornamewithlimits{Res}}

\renewcommand{\L}{\mathcal{L}}

\newcommand{\kk}{\widehat{k}}
\newcommand{\lala}{\widehat{\lambda}}

\newcommand{\fw}{\mathcal{F}_\Phi}
\newcommand{\p}{\prime}
\newcommand{\pp}{\prime\prime}
\newcommand{\rr}{{\Phi}}

\newcommand{\Bases}{\mathcal{B}}
\newcommand{\bb}{\mathbf{B}}

\newcommand{\DD}{\mathcal{D}}

\newcommand{\iber}{\operatornamewithlimits{iBer}}

\newcommand{\tree}{\mathrm{Tree}}
\newcommand{\ch}{\operatorname{ch}}

\newcommand{\affweyl}{\widetilde{\Sigma}[{k}]}

\newcommand{\alphalink}{\beta_{\mathrm{link}}}

\usepackage{stackengine}

\newcommand{\bes}{\begin{eqnarray*}}
	\newcommand{\ees}{\end{eqnarray*}}
\newcommand{\beq}{\begin{eqnarray}}
	\newcommand{\eeq}{\end{eqnarray}}

 \theoremstyle{plain}
 \newtheorem{theorem}{Theorem}[section]
\newtheorem{proposition}[theorem]{Proposition}
 \newtheorem{lemma}[theorem]{Lemma}
 
\newtheorem{corollary}[theorem]{Corollary}
\theoremstyle{definition}
 
\newtheorem{example}{Example}

 \newtheorem{remark}[theorem]{Remark}

\title[Tautological bundles on parabolic moduli spaces]{Tautological bundles on parabolic moduli spaces: Euler characteristics and Hecke correspondences}
\author{Olga Trapeznikova}
\address{Section de mathématiques, Université de Genève}
\email{Olga.Trapeznikova@unige.ch}
\begin{document}	
\begin{abstract} 
We calculate the Euler characteristic of  associated vector bundles over the moduli spaces of stable parabolic bundles on smooth curves. Our method is based on a wall-crossing technique from Geometric Invariant Theory,  certain iterated residue calculus and the tautological Hecke correspondence. Our work was motivated by the results of Teleman and Woodward on the index of K-theory classes on moduli stacks.
\end{abstract}

\maketitle

\begin{section}{Introduction}\label{intro}
Let $C$ be a smooth, complex projective curve of genus $g\geq2$, and fix a point $p\in C$. Denote by  $\Delta$ the set of vectors $c=(c_1>c_2>...>c_r)\in \mathbb{R}^r$ such that  $\sum c_i=0 $ and $c_1-c_r<1 $. We will call a vector $c\in\Delta$ \textit{regular} if  no nontrivial 	subset of its coordinates sums to an integer. For such a $ c\in\Delta$,
	there exists a smooth projective moduli space $P_0(c)$ of dimension $(r^2-1)(g-1)+{r \choose 2}$
	(\cite{Seshadri, MehtaSeshadri, Bhosle}), whose points are in one-to-one correspondence 
	with equivalence classes of pairs $ 
	(W,F_*) $, where $ W$ is a vector bundle  of rank $ r $ 
	on $ C $ with trivial determinant, $ F_* $ is a full flag 
	in the fiber $ W_p $, and the pair satisfies a certain  
	parabolic stability condition depending on a regular $ c\in\Delta$. 
	
	There is a natural way to associate to an integer $k>0$ and a vector $\lambda\in\mathbb{Z}^r$ satisfying $\lambda_1+...+\lambda_r=0$  a line bundle $\mathcal{L}(k;\lambda)$ on $P_0(c)$ in such a way that if $c=\lambda/k$, then $\L(k;\lambda)$ is ample.
	 
\noindent\textbf{Notation:} Let $V=\R^r/(1,...,1)\R$ thought as the Cartan subalgebra of the Lie algebra of $SU_r$;  we denote by $\rho=\frac{1}{2}(r-1,r-3,...,-r+1)$ the half-sum of positive roots of $SU_r$, and set  $ \lala =\lambda+\rho $, $\kk=k+r$. We introduce the notation  $w_\Phi=\prod_{i<j} \left( 2\mathrm{sinh}(x_i-x_j)\right)$ for the Weyl denominator, a function on  $V\otimes_\R\C$.

In \cite{OTASz} we gave a new proof of the parabolic Verlinde formula for the Euler characteristic of the line bundle $\mathcal{L}(k;\lambda)$ on $P_0(c)$. In fact, we proved a more precise formula, which has the following form:
\begin{equation} \label{thmver}
	\chi(P_0({c}),\L({k;\lambda}))
	=   N \cdot\sum_{\bb\in\DD} \int\displaylimits_{Z_\bb}  	w_\Phi^{1-2g}(x/\kk)\exp\scp{\lala}{x}\;{OS}_{\bb,c},
\end{equation}
where  the sum runs over some finite set $\DD$, $N$ is a constant, which depends on $g$,  $r$ and $k$, 
 $Z_\bb$ is a cycle near the origin in $V\otimes_\R\C\setminus\{w_\Phi=0\}$ and ${OS}_{\bb,c}$ is a differential form on $V\otimes_\R\C\setminus\{w_\Phi=0\}$, which depends on $c$  and $\bb$ (for details see page \pageref{defiber}).

In this paper, we present a formula for the Euler characteristic of a wider class of vector bundles on $P_0(c)$: we associate to  a dominant weight $\nu$ of $GL_r$ a tautological vector bundle $U_\nu$ on $P_0(c)\times C$ and calculate the Euler characteristic 
\begin{equation}\label{goalintro}
\chi(P_0({c}),\L({k;\lambda})\otimes\pi_!({U_\nu}\otimes\mathcal{K}^{\frac{1}{2}})),
\end{equation}
where $ \pi: P_0(c)\times C\to P_0(c)$ is the projection, and $ \mathcal{K} $ is the canonical bundle on $ C $. Our main result is Theorem \ref{main} (cf. Example \ref{ex:2form} and \eqref{rank2final} for some examples).

Our proof follows the strategy of \cite{OTASz}, whose basic ideas we recall now.
The simplex $\Delta$  of \textit{parabolic weights} $c$  parametrizing stability conditions contains a finite number of hyperplanes (\textit{walls})  on whose complement (the set of regular elements in $\Delta$) the stability condition is locally constant. This induces a \textit{chamber} structure on $\Delta$, such that the left-hand side and the right-hand side of \eqref{thmver} are manifestly polynomial in the variables $(k;\lambda)$ on each chamber. We introduce the notation $l_c(k;\lambda)$ and $r_{c}(k; \lambda)$ for these polynomials, where $c$ is any element of the corresponding chamber. 

Using geometric invariant theory, we show that the wall-crossing terms, i.e. the differences between the two polynomials associated to neighbouring chambers (specified by $c_+$ and $c_-$) for the left-hand side and the right-hand side coincide: 
\begin{equation}\label{diffintro}
l_{c_+}-l_{c_-}=r_{c_+}-r_{c_-}.
\end{equation}

Next, we choose a pair of  chambers  adjacent to two special vertices of the simplex $\Delta$, and consider the corresponding pairs of polynomials: 
\begin{equation}\label{pairintro}
l_{c_>}(k;\lambda), \, l_{c_<}(k;\lambda) \text{\, \, and \, \, } r_{c_>}(k;\lambda), \, r_{c_<}(k;\lambda)
\end{equation}
  from the left-hand side  and the right-hand side of \eqref{thmver}, respectively. Using Serre duality, we derive certain symmetry properties of $l_{c_>}(k;\lambda)$ and $l_{c_<}(k;\lambda)$, and then we prove that $r_{c_>}(k;\lambda)$ and $r_{c_<}(k;\lambda)$ satisfiy the same symmetries.

Finally, we  show that a set of polynomials  parametrized by the chambers in $\Delta$ is uniquely determined by  the wall-crossing terms \eqref{diffintro} and our symmetry properties for the polynomials \eqref{pairintro},
and thus we obtain that $l_c(k;\lambda)$ and $r_c(k;\lambda)$ coincide.
	
In this paper, we follow a similar path.
To demonstrate the technique,   
 below, after our introductory remarks in  \S\ref{S1.1}, we will present our arguments for the case $r=2$, when the formula for the Euler characteristic has a simple form (cf. \eqref{rank2final}).

\textbf{Acknowledgments}. I would like to express my gratitude to my thesis advisor, Andras Szenes, for his guidance, tremendous support, encouragement and help. This research was supported by SNF grant 175799, and the NCCR SwissMAP.

\subsection{Remarks}
In this paragraph, we discuss the relationship of our paper with earlier works.

As mentioned above, we will follow the ideas of \cite{OTASz}, where the case of the line bundles on the moduli spaces was treated. Let us highlight some of the new phenomena that we encountend in this work. 
The symmetry of Euler characteristics \eqref{goalintro} on the moduli spaces $P_0(c_>)$ and $P_0(c_<)$ is only true after an affine transformation; in fact, they need to be  shifted by a linear combination of Euler characteristics of line bundles, which then can be calculated using  the results of \cite{OTASz} (cf. Propositions \ref{geometricsymm}, \ref{propsymmIBer} and \ref{diffshifted}). The appearance of Hessians  (cf. \cite{Wittenrevisited, TelemanW}) in our framework in the formulas for Euler characteristics (cf. Theorem \ref{main}) is remarkably  simply explained by the relations in the cohomology ring of the curve (cf. page \pageref{locjaccalc}). 
The directional derivatives, on the other hand, appear 
from a comparison of the Chern characters of the corresponding vector bundles under the Hecke isomorphism  (cf. Proposition \ref{HeckeChern}).

Our work owes a lot to the paper of Teleman and Woodward \cite{TelemanW}, where a similar formula is derived for Euler characteristics of vector bundles on stacks. The advantage of our approach is its technical simplicity, and much more explicit formulas  when the invariants of moduli spaces are to  be calculated. There are also subtle differences in the final formulas, which are manifest, in particular, in the appearance of certain determinantal factors in our formalism. 

The idea of the formulas for push-forwards in the cohomological setting, in particular,  the Hessian, first appeared in the seminal paper of Witten \cite{Wittenrevisited}. Mathematically sound approaches in this cohomological/symplectic setting  were employed by Jeffrey and Kirwan \cite{JK} and Meinrenken  \cite{Mein}. In particular, the wall-crossing ideas, which play a major role in our work already appeared in \cite{JK}.

Finally, a few words on the significance of our paper:
 we show that the residue/ wallcrossing methods of \cite{OTASz} may be successfully employed to describe the pushforward maps in the full K-theory of moduli spaces. 
The formulas we find, even though they are similar  to the results of \cite{JK} and \cite{TelemanW}, are new, and, in fact, are the first explicit formulas for these quantities.

\subsection{The residue formula for rank 2}\label{S1.1}
In this case we need to consider the moduli space of vector bundles with parabolic structures at \textit{two} points to calculate our wall-crossing terms. For the convenience of the reader, we recall bellow some notations from \cite[\S9]{OTASz}. 
We fix two points: $p,q\in C$, and consider the moduli space 
$$P_0(c,a) = \{W\to C, \, F_1\subset W_p, \, G_1\subset W_q | \,  \mathrm{rk}(W)=2,\, \mathrm{det}(W)\simeq \mathcal{O}\}$$
of rank-2
 stable parabolic bundles $W$ 
 with fixed determinant isomorphic to $\mathcal{O}$, with parabolic 
 structure given by a line $F_1\subset W_p$ with weight $(c, -c)$, and a line
 $G_1\subset W_q$ with weight $(a, -a)$.
 
 The space of admissible parabolic weights is a square 
 $$\Box = \{(c, a) \, |  \, 1>2c>0, \, 1>2a>0\, \};$$
 it is easy to check that the set of isomorphism classes of parabolic bundles remains unchanged when we vary the parabolic weights in
each of the two chambers defined by the conditions 
$$c>a \text{ \, and \, } c<a.$$
There are thus two moduli spaces, which we denote by  $P_0(c>a)$ and $P_0(c<a)$.
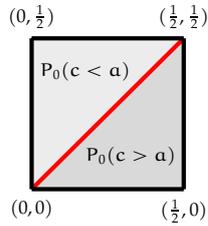
\begin{figure}[H]
	\centering
	\begin{tikzpicture}[scale=2]
		\draw [fill=lightgray!30] (0,0) -- (0,1) -- (1,1);
		\draw [fill=gray!30] (0,0) -- (1,0) -- (1,1);
		\draw  [ultra thick, red] (0,0) -- (1,1);
		\draw  [ultra thick] (0,0) -- (0,1);
		\draw  [ultra thick] (0,0) -- (1,0);
		\draw  [ultra thick] (1,0) -- (1,1);
		\draw  [ultra thick] (0,1) -- (1,1);
		\node [above] at (0.35,0.65) {\tiny $P_0(c<a)$};
		\node [below] at (0.65,0.35) {\tiny $P_0(c>a)$};
		\node [below] at (0,0) {\tiny $(0,0)$};
		\node [below] at (1,0) {\tiny $(\frac{1}{2},0)$};
		\node [above] at (0,1) {\tiny $(0,\frac{1}{2})$};
		\node [above] at (1,1) {\tiny $(\frac{1}{2},\frac{1}{2})$};
\end{tikzpicture}
\setlength{\belowcaptionskip}{-8pt}\caption{The space of admissible weights in the case of rank $r=2$, two points.} \label{square}
\end{figure}

Denote by the same symbol $U$ universal bundles  over $P_0(c>a)\times C$ and $P_0(c<a)\times C$. Then $U$ is endowed with  two flags, $\mathcal{F}_1\subset \mathcal{F}_2=U_p$ and $\mathcal{G}_1\subset \mathcal{G}_2=U_q$; we  choose a  normalization of $U$ such that $\mathcal{F}_1\subset U_p$ is trivial.
For $\lambda, \mu \in\mathbb{Z}$, we 
introduce the line bundle 
\begin{equation*}
\begin{split}
\mathcal{L}(k; \lambda, {\mu})= & 
\mathrm{det}({U}_p)^{k(1-g)}\otimes\mathrm{det}(\pi_*(U))^{-k} 
\otimes(\mathcal{F}_2/\mathcal{F}_{1})^{\lambda}
\otimes(\mathcal{F}_1)^{-\lambda}
\otimes(\mathcal{G}_2/\mathcal{G}_1)^{\mu}\otimes(\mathcal{G}_1)^{-\mu}
\end{split}
\end{equation*}
on  $P_0(c>a)\times C$ and $P_0(c<a)\times C$.

Let $\nu=(\nu_1\geq \nu_2)\in \mathbb{Z}^2$ be a dominant weight of $GL_2$, denote by $ \rho_\nu$  the irreducible representation of $GL_2$ with highest weight $\nu$, and by $\bar{\rho}_\nu$ its restriction to $SU_2\subset GL_2$. We denote by $U_\nu\to P_0(c,a)\times C$ the bundle associated to the representation $\rho_\nu$.

Our goal is to calculate  Euler characteristics 
\begin{equation}\label{intropoly}
\begin{split}
\chi^\nu_>(k;\lambda,\mu)\overset{\mathrm{def}}=\chi(P_0(c>a),\mathcal{L}{(k;\lambda,\mu)}\otimes\pi_!({U_\nu}\otimes\mathcal{K}^{\frac{1}{2}})) \text{\,\,\, and\,\,\,} \\
\chi^\nu_<(k;\lambda,\mu)\overset{\mathrm{def}}=\chi(P_0(c<a),\mathcal{L}{(k;\lambda,\mu)}\otimes\pi_!({U_\nu}\otimes\mathcal{K}^{\frac{1}{2}})).
\end{split}
\end{equation}
Let $\mathrm{Exp}: \mathrm{Lie}(SU_2) \to SU_2 $ be the exponential map and let $$\phi(x)=\mathrm{trace}(\bar{\rho}_\nu\circ\mathrm{Exp}(x/2))=\frac{\mathrm{sinh}((\nu_1-\nu_2+1)x/2)}{\mathrm{sinh}(x/2)}$$ be the character function on the Lie algebra of a maximal torus of $SU_2$.

We introduce the notation 
$$\dot{\phi}(x)= 2\frac{d}{d x} \phi(x) \text{\, \, and \, \,}  \ddot{\phi}(x)= 2\frac{d}{d x} \dot{\phi}(x),$$ (where the factors of $2$ are introduced for convenience) and define two polynomials in $k, \lambda, \mu$ which as we will show, equal to \eqref{intropoly}:
\begin{multline*}
R^\nu_>(k;\lambda,\mu)\overset{\mathrm{def}}=(-1)^{g}\frac{\partial }{\partial \delta}\Big{|}_{\delta=0}  \\ \frac1{(2\pi i)} \int\displaylimits_{|u|=\varepsilon}\frac{(e^{u(\lambda+\mu+1)}
	-e^{u(\lambda-\mu)})e^{u(\nu_1+\nu_2)/2}(2k+4+\delta\ddot{\phi}(u))^g}{(2\mathrm{sinh}\left(\frac{u}{2}\right))^{2g}(1-e^{u(k+2)+\delta\dot{\phi}(u)})}du
\end{multline*}
and
\begin{multline*}
R^\nu_<(k;\lambda,\mu)\overset{\mathrm{def}}=(-1)^{g}\frac{\partial }{\partial \delta}\Big{|}_{\delta=0}  \\ \frac1{(2\pi i)} \int\displaylimits_{|u|=\varepsilon}\frac{(e^{u(\lambda+\mu+1)}
	-e^{u( \lambda-\mu+k+2)+\delta\dot{\phi}(u)})e^{u(\nu_1+\nu_2)/2}(2k+4+\delta\ddot{\phi}(u))^g}{(2\mathrm{sinh}\left(\frac{u}{2}\right))^{2g}(1-e^{u(k+2)+\delta\dot{\phi}(u)})}du,
\end{multline*}	
where $\varepsilon$ is a real constant and $\delta\ll\varepsilon$.
We note two facts about this pair of polynomials:
\begin{enumerate}
		\item[Fact 1.]\label{polytwofacts} The  difference of these polynomials has the form:
$$R^\nu_>(k;\lambda,\mu)-R^\nu_<(k;\lambda,\mu) = g(-(2k+4))^{g-1}  \res_{\substack{u=0}}\frac{e^{u(\lambda-\mu)}e^{u(\nu_1+\nu_2)/2}\ddot{\phi}(u)}{(2\mathrm{sinh}\left(\frac{u}{2}\right))^{2g}}du.$$
		\item[Fact 2.]  An easy calculation via substitutions shows the following: 
\begin{multline*}
R^\nu_>(k;\lambda,\mu)=-R^\nu_>(k;\lambda,-\mu-1)= -R^\nu_>(k;-\lambda+k+1-(\nu_1+\nu_2),\mu)- \\ (-(2k+4))^g  \res_{\substack{u=0}}\frac{(e^{u(\lambda+\mu+1)}
	-e^{u(\lambda-\mu)})e^{u(\nu_1+\nu_2)/2}\dot{\phi}(u)}{(2\mathrm{sinh}\left(\frac{u}{2}\right))^{2g}(1-e^{u(k+2)})}du
	\end{multline*}
	and 
\begin{multline*}
R^\nu_<(k;\lambda,\mu)=	-R^\nu_<(k;-\lambda-1-(\nu_1+\nu_2),\mu)=-R^\nu_<(k;\lambda,-\mu+k+1)- \\ (-(2k+4))^g  \res_{\substack{u=0}}\frac{(e^{u(\lambda+\mu+1)}
	-e^{u(\lambda-\mu+k+2)})e^{u(\nu_1+\nu_2)/2}\dot{\phi}(u)}{(2\mathrm{sinh}\left(\frac{u}{2}\right))^{2g}(1-e^{u(k+2)})}du.
	\end{multline*}
\end{enumerate}
\subsection{Hecke correspondences, Serre duality and the symmetry argument}
In this section, we prove that the polynomials $\chi^\nu_>(k;\lambda,\mu)$ and 
$\chi^\nu_<(k;\lambda,\mu)$ (cf.  \eqref{intropoly}) satisfy the same antisymmetries as the polynomials $R^\nu_>$ and $R^\nu_<$ (cf. Fact 2).

In \cite[\S7.1]{OTASz} (c.f. also \S\ref{S4.4}) we describe the tautologial variant of the Hecke correspondence which identifies the moduli spaces of parabolic bundles with different degrees and weights. 
Applying the Hecke correspondence at the point $p$ and $q$ to $P_0(c>a)$ and $P_0(c<a)$ respectively, we can identify these spaces as $\mathbb{P}^1\times\mathbb{P}^1$-bundles over the moduli space $N_{-1}$ of stable bundles of degree $-1$ (cf. \cite[Lemma 9.3]{OTASz}):
$$\mathbb{P}^1\times\mathbb{P}^1 \rightarrow P_0(c>a) \rightarrow N_{-1} \leftarrow P_0(c<a) \leftarrow \mathbb{P}^1\times\mathbb{P}^1.$$
For each copy of $\mathbb{P}^1$, the moduli space  $P_0(c>a)$ can be considered as a $\mathbb{P}^1$-bundle; applying  Serre duality for families of curves as in \S\ref{S5.1}, we obtain the following two equalities:
$$ \chi^\nu_>(k;\lambda,\mu) = -\chi^\nu_>(k;\lambda,-\mu-1) $$ and 
\begin{multline*} \chi^\nu_>(k;\lambda,\mu) = -\chi^\nu_>(k;-\lambda+k+1-(\nu_1+\nu_2),\mu)+ \\ \sum_{i=0}^{\nu_1-\nu_2} (\nu_1-\nu_2-2i)\chi(P_0(c>a), \mathcal{L}(k; \lambda+\nu_1-i, \mu)).
\end{multline*}
Similarly, for $P_0(c<a)$ we obtain that 
\begin{multline*}
 \chi^\nu_<(k;\lambda,\mu) = -\chi^\nu_<(k;-\lambda-(\nu_1+\nu_2)-1,\mu) = -\chi^\nu_<(k;\lambda,-\mu+k+1)+ \\ \sum_{i=0}^{\nu_1-\nu_2} (\nu_1-\nu_2-2i)\chi(P_0(c>a), \mathcal{L}(k; \lambda+\nu_1-i, \mu)).
\end{multline*}
We showed in \cite[\S9]{OTASz} that 
\begin{multline*}
\sum_{i=0}^{\nu_1-\nu_2} (\nu_1-\nu_2-2i)\chi(P_0(c>a), \mathcal{L}(k; \lambda+\nu_1-i, \mu)) =  \\ (-1)^{g-1}(2k+4)^g\res_{\substack{u=0}}\frac{(e^{u(\lambda+\mu+1)}
	-e^{u(\lambda-\mu)})e^{u(\nu_1+\nu_2)/2}\dot{\phi}(u)}{(2\mathrm{sinh}\left(\frac{u}{2}\right))^{2g}(1-e^{u(k+2)})}du
	\end{multline*}
	 and 
\begin{multline*}
\sum_{i=0}^{\nu_1-\nu_2} (\nu_1-\nu_2-2i)\chi(P_0(c<a), \mathcal{L}(k; \lambda+\nu_1-i, \mu)) = \\  (-1)^{g-1}(2k+4)^g  \res_{\substack{u=0}}\frac{(e^{u(\lambda+\mu+1)}
	-e^{u(\lambda-\mu+k+2)})e^{u(\nu_1+\nu_2)/2}\dot{\phi}(u)}{(2\mathrm{sinh}\left(\frac{u}{2}\right))^{2g}(1-e^{u(k+2)})}du, 
\end{multline*}
hence the polynomials $\chi^\nu_>$ and $\chi^\nu_<$ satisfy the same antisymmetries as $R^\nu_>$ and $R^\nu_<$ (cf. Fact 2 on page \pageref{polytwofacts}).

\subsection{Wall-crossing in moduli spaces}\label{S1.3}
Our next step is to  compare the difference $\chi^\nu_>-\chi^\nu_<$ with the difference $R^\nu_> - R^\nu_<$ from Fact 1 on page \pageref{polytwofacts}. 

In \cite[\S5]{OTASz} we presented a simple formula for the wall-crossing difference in Geometric Invariant Theory. The formula has the form of a residue of an equivariant integral, taken with respect to an equivariant parameter. In the rank-2 case (cf. \cite[\S9]{OTASz}), the space $Z^0$ over which we integrate is isomorphic to the Jacobian of degree-0 line bundles on $C$:
$$Z^0 \simeq \{V=L\oplus{L^{-1}} \, |\, L\in \mathrm{Jac}^0, \, L_p= F_1,  L_q^{-1}= G_1\}.$$
We thus obtain the following expression for the wall-crossing difference: 
\begin{multline}\label{locjac}
\chi^\nu_>(k;\lambda,\mu) - \chi^\nu_<(k;\lambda,\mu) = \\  (-1)^{g}\res_{u=0}\frac{\exp(u(\lambda-\mu))}{(2\mathrm{sinh}(u/2))^{2g}}\int_{\mathrm{Jac}} e^{\eta(2k+4)}ch(\pi_!(U_\nu\otimes{\mathcal{K}}^{\frac{1}{2}})\big{|}_{\mathrm{Jac}})\, du,
\end{multline}
where $u$ plays the role of the equivariant parameter, the generator of $H^*_{\mathbb{C}^*}(pt)$; let $\mathcal{J}$ be the Poincare bundle over $\mathrm{Jac}\times C$, 
satisfying  $c_1(\mathcal{J}_p)=0$, then  the class $\eta\in H^2(\mathrm{Jac})$ is defined through the K\"unneth decomposition of $c_1(\mathcal{J})^2$  (cf. page \pageref{zagjac}).

It follows from the Groethendieck-Riemann-Roch theorem that 
$$ch(\pi_!(U_\nu\otimes{\mathcal{K}}^{\frac{1}{2}})) = \pi_*ch(U_\nu).$$
A simple calculation shows that  the restriction $U\big{|}_{Z^0}=\mathcal{J}\oplus\mathcal{J}^{-1}$ has $\mathbb{C}^*$-weight 1, hence  we have $$ch(U_\nu\big{|}_{Z^0})=\bigoplus_{i=0}^{\nu_1-\nu_2}ch(\mathcal{J}^{\nu_1-\nu_2-2i})\exp({(\nu_1-i)u}).$$
Note that that 
$\pi_*(ch(\mathcal{J}^n))=-n^2\eta$, and thus $$\pi_*(ch(U_\nu\big{|}_{Z^0}))=-\eta \sum_{i=0}^{\nu_1-\nu_2}(\nu_1-\nu_2-2i)^2\exp({(\nu_1-i)u})=-\eta \exp((\nu_1+\nu_2)u/2)\ddot{\phi}(u).$$
Using \eqref{zagjac}, we obtain that the wall-crossing difference \eqref{locjac} is equal to 
\begin{equation}\label{locjaccalc}
g(-(2k+4))^{g-1}\res_{u=0}\frac{\exp(u(\lambda-\mu))e^{u(\nu_1+\nu_2)/2}\ddot{\phi}(u)}{(2\mathrm{sinh}(u/2))^{2g}}\, du,
\end{equation}
and thus we have (cf. Fact 1 on page \pageref{polytwofacts})
\begin{equation}\label{RminX}
R^\nu_>-R^\nu_< =\chi^\nu_>-\chi^\nu_<.
\end{equation}

Now we are ready for the final argument: we can rearrange equation \eqref{RminX} to describe the equality of wall-crossings as 
\begin{equation}\label{RminXarr}
R^\nu_>(k;\lambda,\mu)-\chi^\nu_>(k;\lambda,\mu)= R^\nu_<(k;\lambda,\mu)-\chi^\nu_<(k;\lambda,\mu);
\end{equation}
we introduce the notation $\Theta(k;\lambda,\mu)$ for this polynomial. 
Then $\Theta(k;\lambda,\mu)$ satisfies 4 antisymmetries:
\begin{multline*}
\Theta(k;\lambda,\mu) = -\Theta(k;\lambda,-\mu-1)=-\Theta(k;-\lambda+k+1-(\nu_1+\nu_2),\mu)= \\ -\Theta(k;-\lambda-1-(\nu_1+\nu_2),\mu)=-\Theta(k;\lambda,-\mu+k+1),
\end{multline*}
hence it is anti-invariant with respect to the affine Weyl group action on $\lambda$ and $\mu$ separately, and this implies $\Theta=0$. 

As $P_0(c>a)$ is a $\mathbb{P}^1$-bundle over the moduli space of rank-2 
degree-0 stable parabolic bundles $P_0(c)$,  substituting $\mu=0$ in $R^\nu_>$ and taking the derivative with respect to $\delta$, we obtain the formula for rank 2:
\begin{multline}\label{rank2final}
\chi(P_0({c}),\L({k;\lambda})\otimes\pi_!({U_\nu}\otimes\mathcal{K}^{\frac{1}{2}}))= 
 \\(-(2k+4))^{g}\res_{u=0}\frac{\exp(u(\lambda+\frac{1}{2}+\frac{\nu_1+\nu_2}{2}))}{(2\mathrm{sinh}\left(\frac{u}{2}\right))^{2g-1}(1-e^{u(k+2)})}\left(\frac{g\, \ddot{\phi}(u)}{2k+4}+\frac{e^{(2k+4)u}\dot{\phi}(u)}{ (1-e^{u(k+2)})}\right) du.
\end{multline}

\begin{subsection}{Contents of the paper}
The paper is organized as follows. We start with a quick overview of the theory of parabolic bundles in \S\ref{S2.1}-\S\ref{S2.11}; here we describe the vector bundles we are considering and introduce the chamber structure on the space of parabolic weights. In \S\ref{S2.2}-\S\ref{S2.3} we briefly recall the notion of diagonal bases of hyperplane arrangements first introduced in \cite{Szimrn}. 
Using this object, in \S\ref{S3.1} we present our main result, Theorem \ref{main}. The proof of this theorem takes up the rest of the paper. 

In \S\ref{S3.2} (cf. Corollary \ref{corwcresidue})  we calculate the wall-crossing difference in residue formulas (general version of Fact 1 on page \pageref{polytwofacts}). In \S\ref{S4} we apply the formula for wall-crossings in GIT  \cite[Theorem 5.6]{OTASz} to the moduli space of parabolic bundles with different parabolic weights. We obtain Proposition \ref{wcrprop}, the higher rank version of formula \eqref{locjaccalc} above. 

In \S\ref{S5.1} we derive Weyl antisymmetries for the polynomials $\chi^\nu_>$, $\chi^\nu_<$ and in \S\ref{S5.2}-\S\ref{S5.3} we show  the same antisymmetries for the polynomials $R^\nu_>$, $R^\nu_<$. In \S\ref{S6.1} we finish the proof following the idea described above in \S\ref{S1.3}. We end the paper with  a mild generalizations (cf. \S\ref{S6.2}) of our main result. 
\end{subsection}

\end{section}	
\begin{section}{Preliminaries}

\subsection{Parabolic bundles}\label{S2.1}
In this section, we briefly review the definition of parabolic bundles, repeat the basic facts about their moduli spaces from \cite[\S2]{OTASz} and  describe the chamber structure on the space of the relevant parameters, known as \textit{parabolic weights}.

Let $C$ be a smooth complex projective curve of genus $g\geq2$, fix a point $p\in{C}$ and a positive integer $r$. 
A \textit{parabolic bundle} on $C$ is a vector bundle $W$ of rank $r$ equipped with  a full flag  $ F_* $ in the fiber over $p$:
$$W_p=F_r\supsetneq...\supsetneq F_1\supsetneq F_0=0,$$  and   \textit{parabolic 
weights} $c=(c_1,...,c_r)$ assigned to $F_r, F_{r-1},..., F_1$,
 satisfying the conditions
$$c_1 > c_{2}> . . . > c_r \text{ and } c_1-c_r<1.$$ The \textit{parabolic degree} of $W$ is defined as $$\mathrm{pardeg}(W)=\mathrm{deg}(W)-\sum_{i=1}^{r}c_i.$$
Any subbundle $W'$ of a parabolic bundle $W$  and the corresponding quotient $W/W'$ inherit a parabolic structure  in a natural way (cf. \cite{MehtaSeshadri}, definition 1.7).
 
The parabolic bundle $W$ is \textit{stable}, if any proper parabolic subbundle $W'\subset W$ satisfies 
$$\frac{\mathrm{pardeg}(W')}{\mathrm{rk}(W')}<\frac{\mathrm{pardeg}(W)}{\mathrm{rk}(W)}.$$

Note that the parabolic stability condition depends on parabolic weights only up to adding the same constant to all weights $c_i$, so we can assume that for fixed rank $r$ and degree $d$, the space of all values for the weights $c$ is the simplex 
$$\Delta_{d}= \left\{(c_1,c_{2}, ..., c_r) \, | \, c_1>{c_{2}}> ... >{c_r}, \, c_1-c_r<1, \, \sum_ic_i=d\right\}.$$ \label{Delta}

\noindent \textbf{Definition:}\label{regular} We will call a vector $ c=(c_1,\dots,c_r)\in\R^r $ such that $ \sum_ic_i\in\Z $ 
\textit{regular} if for any nontrivial subset $ \Psi\subset\{1,2,\dots,r\} $, 
we have $ \sum_{i\in\Psi}c_i\notin\Z $.

For fixed rank $r$, degree $d$ and regular $c=(c_1,...,c_r)\in \Delta_d$, Mehta and Seshadri \cite{MehtaSeshadri} constructed  a smooth projective moduli space  of stable parabolic bundles $\tilde{P}_d(c)$, whose points are in one-to-one correspondence with the set of isomorphism classes of stable parabolic bundles of weight $c$.

Via the determinant map, the moduli space $\tilde{P}_d(c)$ fibers over the Jacobian of degree-d line bundles on $C$ with isomorphic fibers. In this paper, we will focus on these fibers, the moduli space
\[ P_d(c) =  \{W\in\widetilde{P}_d(c)|\; \det W \simeq\mathcal{O}(d{p}) \},\]
which is smooth and projective of dimension $(r^2-1)(g-1)+{r \choose 2}$.

In \cite[\S2.4]{OTASz} we described a set of affine hyperplanes in $\Delta_d$, called \textit{walls},  parametrized by a nontrivial partition $\Pi=(\Pi',\Pi'')$ of the first $r$ positive  integers, and a pair of numbers $d', d''$, such that $d'+d''=d$. We have shown that $\Delta_d$ is separated by these walls into a finite number of \textit{chambers}, such that the moduli spaces $P_d(c)$ remain unchanged when varying $c$ within a chamber. 
\begin{example}\label{ex:wall}
Consider the case of rank-3 degree-0 stable parabolic bundles with parabolic 
weights $c=(c_1,c_2,c_3)\in\Delta$. As observed in \cite[Example 1]{OTASz}, in this case $\Delta$ is an open triangle (cf. Figure \ref{fig:triang}) and 
there exist only two different stability conditions. The wall 
separating the two chambers is given by the condition $c_2=0$. We write $ P_0(>) 
$ for the moduli space
$P_0(c_1,c_2,c_3)$  with $c_2>0$, and $ P_0(<) $ for
$P_0(c_1,c_2,c_3)$ with  $c_2<0$.
\end{example}
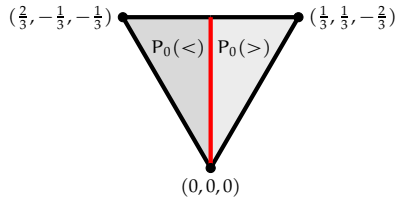
\begin{figure}[H]
\centering
\begin{tikzpicture}[scale=2]
\draw [ultra thick , fill=lightgray!30] (0,0) -- ({1/sqrt(3)},1) -- (0,1);
\draw [ultra thick , fill=gray!30] (0,0) -- ({-1/sqrt(3)},1) -- (0,1);
\draw  [red, ultra thick] (0,0) -- (0,1);
\node [below] at (0,0) {\tiny $(0,0,0)$};
\node [left] at ({-1/sqrt(3)},1) {\tiny $(\frac{2}{3},-\frac{1}{3},-\frac{1}{3})$};
\node [right] at ({1/sqrt(3)},1) {\tiny $(\frac{1}{3},\frac{1}{3},-\frac{2}{3})$};
\draw [fill] (0,0) circle [radius=0.03];
\draw [fill] ({-1/sqrt(3)},1) circle [radius=0.03];
\draw [fill] ({1/sqrt(3)},1) circle [radius=0.03];
\node [above] at ({sqrt(3)/8},{2/3}) {\tiny $P_0(>)$};
\node [above] at ({-sqrt(3)/8},{2/3}) {\tiny $P_0(<)$};
\end{tikzpicture}
\setlength{\belowcaptionskip}{-8pt}\caption{The space of admissible parabolic weights for  $r=3$.}\label{fig:triang}
\end{figure}

\subsection{Vector bundles on the moduli space of parabolic bundles}\label{S2.11}
For a regular $c\in\Delta_d$, there exists a universal bundle $U$ over $P_d(c)\times C$, endowed with a flag $ \mathcal{F}_{1}\subset\dots\subset\mathcal{F}_{r-1}\subset\mathcal{F}_{r}=U_p$, and satisfying the obvious tautological properties.
Such universal bundle $U$, and hence the flag line bundles 
$\mathcal{F}_{i+1}/\mathcal{F}_i$ are unique only up to tensoring by the 
pull-back of a line bundle from $P_d(c)$.   

\noindent \textbf{Definition:}\label{normalized} We will say that  $U$ is \textit{normalized} if 
the line subbundle  $\mathcal{F}_1\subset U_p$ is trivial.

For $k\in\mathbb{Z}$ and $\lambda=(\lambda_1,...,\lambda_r)\in\mathbb{Z}^r$, such that $\sum_{i=1}^r\lambda_i=kd$, we define the line bundle 
\begin{equation*}
\mathcal{L}_d(k; {\lambda})=\mathrm{det}(U_p)^{k(1-g)}\otimes\mathrm{det}(\pi_*U)^{-k}  \otimes(\mathcal{F}_r/\mathcal{F}_{r-1})^{\lambda_1}\otimes...\otimes(\mathcal{F}_1)^{\lambda_r}
\end{equation*} 
on  $P_d(c)$. It is easy to check that this line bundle is independent of the choice of the universal bundle $ U $. 

\noindent \textbf{Notation:} In this paper we will mostly consider degree-0 parabolic bundles, so for $d=0$, we will omit the index $d$ from the line bundle $\mathcal{L}(k;\lambda)$ and the space of parabolic weights $\Delta$. 

Let $\nu = (\nu_1,....,\nu_r)$ be a dominant weight  of $GL_r$, consider the irreducible representation $\rho_\nu$ with highest weight $\nu$, and denote by $\bar{\rho}_\nu$ its restriction to the subgroup $SU_r\subset GL_r$. We denote by $\phi^\nu$ the character $\phi^\nu = \mathrm{trace}(\bar{\rho}_\nu\circ \mathrm{Exp})$ on the Lie algebra $V$ of a maximal torus $T\subset SU_r$. We collect our maps on the following diagram. \label{diagram}
\begin{center}
	\begin{tikzcd}
		&  GL_r \arrow[r, "\rho_\nu"] & GL(V_\nu) \\
		V \arrow[r, "\mathrm{Exp}"] & T\subset SU_r \arrow[ru, "\bar{\rho}_\nu"] \arrow[u, hook]
	\end{tikzcd}
\end{center}
Given a representation $\rho_\nu$ of $GL_r$, we denote by $U_\nu$ the vector bundle over $P_0(c)\times C$ associated to the principal $GL_r$-bundle. 

The vector bundle $U_\nu$ has the following explicit construction. Let $U$ be the normalized universal bundle on $P_0(c)\times{C}$, and consider the full flag bundle $\mathrm{Flag}(U) \overset{f}\to P_0(c)\times{C}$. Denote by $L_1,...,L_r$ the standard quotient line bundles on $\mathrm{Flag}(U)$. Then 
\begin{equation}\label{constructionUnu}
U_\nu = f_*(L_1^{\nu_1}\otimes L_2^{\nu_2}\otimes...\otimes L_r^{\nu_r}).
\end{equation}

\begin{remark}\label{flagsection}
Note that the vector bundles $\mathcal{F}_r,...,\mathcal{F}_1$ on the moduli space $P_0(c)$ define a section of the flag bundle $\mathrm{Flag}(U_p)\to P_0(c)\times\{p\}$.
\end{remark}

\subsection{Notation}\label{S2.2}
Following \cite{OTASz}, we set up some extra notation for the space of parabolic weights. 
\begin{itemize}
\item We represent the Cartan subalgebra $V=\mathrm{Lie}(T)$ of the Lie algebra $\mathrm{Lie}(SU_r)$ as the quotient vector space 
 $$ V=\R^r/\R(1,1,\dots,1). $$ 
 There is a natural pairing between $V$ and 
\[ V^*=\{a=(a_1,\dots,a_r)\in\R^r|\; a_1+\dots+a_r=0\}. \]

Let $ x_1,x_2,\dots,x_r $ be the coordinates on $ \R^r $; 
given $ a\in V^* $, we will write $\scp ax  $ for the 
linear function $ \sum_ia_ix_i $ on $ V $. We will 
sometimes denote this linear function simply by $ a 
$.
\item Let $\Lambda$  be the integer lattice  in the vector space $V^*$:
\[ \Lambda=\{\lambda=(\lambda_1,\dots,\lambda_r)\in\Z^r|\; 
\lambda_1+\dots+\lambda_r=0\}. \]
For  $ 1\le i\neq j\le r $, we define 
the 
element $ \alpha^{ij} =x_i-x_j $ in $ \Lambda $. Let
\[ \rr = \{\ \pm\alpha^{ij}  |\,1\le i< j\le r\}  \]
be the set of roots of the $A_{r-1}$ root system with the opposite roots identified. 
Note that  the permutation group  $\Sigma_r$ acts on the vector space $V^*$, permuting the coordinates $x_i$, and the elements of $\rr$.
\item 
A basic object of our approach is an \textit{ordered} linear basis $ \mathbf{B} $  of $ V^* $ consisting of the elements of $ \Phi$. Let us denote the set of these objects by $ \Bases $:
\[ \Bases= \left\{\mathbf{B}=\left(\bta1,\dots,\bta{r-1}\right) \in\Phi^{r-1}|\; \bb \text{ -- basis of }  V^*\right\} \]
\item Given a basis $\mathbf{B}=\left(\bta1,\dots,\bta{r-1}\right) \in \Bases$ of $V^*$, and an element $a\in V^*$, we define $[a]_\mathbf{B}\in\Lambda$ to be the unique element of $V^*$ satisfying $[a]_\mathbf{B} = a- \{a\}_{\bb}$, where  $ \{a\}_{\bb}\in\sum_{j=1}^{r-1}[0,1)\bta j.$
\item We call $a\in V^*$ \textit{regular}, if $ \{a\}_{\bb}\in\sum_{j=1}^{r-1}(0,1)\bta j.$ Then for regular elements $a$ and $b$, we define the equivalence relation 
 \begin{equation*}\label{equivalence}
 	 a\sim b 
 \text{ when }[a]_{\bb} = [b]_{\bb}\quad \forall 
 \bb\in\Bases.
 \end{equation*} 
 
\item Given a partition $ \Pi $ of $ 
\{1,2,\dots, r\} $ into two nonempty sets, we will think of it as an ordered partition $ 
\Pi=(\Pi^{\p},\Pi^{\pp})$  such that $ r\in\Pi^{\pp} $, and we will call these 
objects \textit{nontrivial partitions}.
 \end{itemize} 
\begin{itemize}
\item Then \label{lemmaBwall}\cite[Lemma 4.1]{OTASz}
	the equivalence classes of the relation $ \sim $ are precisely the chambers  
	in $ V^* $ created by the walls parameterized by a nontrivial
	 partition $ \Pi=(\Pi^{\p},\Pi^{\pp})$ of the 
	first $ r $ positive integers, and an integer $l$:
\begin{equation}\label{Swall}
		 S_{\Pi,l} =\{c\in V^*|\;\sum_{j\in\Pi^{\p}}c_j=l\}.
\end{equation}
\end{itemize}

If $d=0$, the set of parabolic weights $\Delta$ defined on page \pageref{Delta} can be considered as an open simplex in $V^*$. Then the intersection of the  walls given in \eqref{Swall} with $\Delta$ are precisely the walls separating the chambers of parabolic weights $c$, in which the moduli spaces $P_0(c)$ of parabolic bundles are naturally the same (cf. end of \S\ref{S2.1}). 

\subsection{Diagonal bases}\label{S2.3}

A key component of our approach is the notion of \textit{diagonal basis} introduced in  \cite{Szimrn}. We refer to \cite{Szimrn}, \cite{Szduke} and \cite[\S3]{OTASz} for basic examples and results on diagonal bases;  now we will briefly recall the combinatorial definition of this object. 
\begin{itemize}
\item If we consider $\rr$ as the edges of the complete graph on $r$ vertices, then the set $\Bases$ is the set of spanning trees of this graph with edges enumerated from $1$ to $r-1$. We denote these ordered trees by  \label{tree}
\[ \bb \mapsto \tree(\bb).\]
\item Given $ \tree(\bb) $, we have a sequence of $ r $ nested 
	partitions of the vertices, which starts with the total 
	partition into 
	1-element sets, ends with the trivial 
	partition, and the $j$th 
	partition is 
	induced by the first $ j-1 $ edges. 
	 A  \textit{ diagonal 
	basis} $ \DD\subset\Bases $ is then a set of $ (r-1)! $ ordered 
	trees such that the $ (r-1)! $ partition sequences 
	obtained by 
	reordering the edges of any one of the ordered trees 
	are different from $ (r-1)!-1 $ sequences of partitions 
	obtained from the remaining elements of $ \DD $. 
\end{itemize}
\begin{example}\label{ex:diag}
 For $r=3$, $\DD = \{(\alpha^{2,3},\alpha^{1,2}), 
(\alpha^{3,2},\alpha^{1,3})\}$ is a diagonal basis. 
\end{example}

\end{section}
\begin{section}{Main result and wall-crossing in residue formulas}
\subsection{Main result}\label{S3.1}
In this section, we formulate our main result, Theorem \ref{main}. 

We introduce 
the notation $ \fw $ for 
the space of meromorphic functions defined in a neighborhood of $ 0 
$ in $ V\otimes_\R \C $ with poles on the union of hyperplanes
\[  \bigcup_{1\le i<j\le r}\{x|\;\scp{\alpha^{ij}}x=0\}.  \] In particular, the inverse of 
$$w_\Phi\overset{\mathrm{def}}=\prod_{i<j} \left( 2\mathrm{sinh}(x_i-x_j)\right)$$ is a function in $\fw$.
Any basis $ \bb\in\Bases $ induces an isomorphism $V^* \simeq V$, and we will write ${\check{\alpha}}_\bb$ for the image of $\alpha\in V^*$ under this isomorphism. We will sometimes omit the index $\bb$ to simplify the notation. For
a function $Q$ on $V$  and $\alpha\in V^*$, we introduce the directional derivative $Q_{\check{\alpha}_\bb}$.
Then, given $Q$ and $ \bb\in\Bases $, we fix a homomorphism 
$$\alpha \mapsto \exp(Q_{\check{\alpha}_\bb})$$ from the additive group $V^*$ to the multiplicative group of non-vanishing holomorphic functions on $V\otimes_{\R}\C$.

Let $K=\frac{1}{2r}\sum_{i<j}\alpha_{ij}^2$ be the normalized Killing form of  $SU_r$ and let $\delta\in \mathbb{R}$ be a small parameter;
given a basis $\bb=\left(\bta1,\dots,\bta{r-1}\right) \in \Bases$ of $V^*$, a function $f\in \fw$ and a holomorphic function $Q=\mathrm{const}\cdot K-\delta\phi$, defined in a neighborhood of $0$ in $V\otimes_{\R}\C$,  we define
  \begin{equation}\label{defiber}
	\iber_{\mathbf{B},Q}   \left[ f(x) \right](a)\overset{\mathrm{def}}=
\frac1{(2\pi i)^{r-1}} \int\displaylimits_{Z_\bb} \frac{ 
	f(x)\exp (Q_{\check{a}})\;d Q_{\check{\beta}^{[1]}}
	\wedge 
	\dots\wedge d{Q}_{\check{\beta}^{[r-1]}}\;}{(1-\exp({Q}_{\check{\beta}^{[1]}}))\;
	\dots(1-\exp({Q}_{\check{\beta}^{[r-1]}}))\;},
\end{equation}
where the naturally oriented cycle $ Z_{\bb} $
is given by
\[ Z_{\bb} = \{x\in 
V\otimes_{\R}\C:\,|\scp{\beta^{[j]}}x|\;=\varepsilon_j,\, 
j=1,\,\dots,r-1 \}\subset 
V\otimes_{\R}\C\setminus\{w_\Phi(x)=0\} \]
with sufficiently small fixed real constants $ \varepsilon_j $ satisfying
$ 0\le\varepsilon_{r-1}\ll\dots\ll \varepsilon_{1} $. Thus $\iber_{\mathbf{B},Q}$ is a linear operator associating to a meromorphic function $f\in\fw$ a polynomial on $V^*$. 

 We  introduce the notation $\mathcal{H}^\Phi$ for the space of holomorphic functions  of the form $Q=\mathrm{const}\cdot K-\delta\phi$, defined in a neighborhood of $0$ in $V\otimes_{\R}\C$. We will always assume that our parameter $\delta$ is small enough, so that the cycle given by $ \{x\in 
V\otimes_{\R}\C:\,|Q_{\check{\beta}^{[j]}}(x)|\;=\varepsilon_j,\, 
j=1,\,\dots,r-1 \}\subset 
V\otimes_{\R}\C\setminus\{w_\Phi(x)=0\}$ is homotopic to the cycle $Z_\bb$.

\noindent\textbf{Notation:} We will write $\iber_{\mathbf{B}}$ for $\iber_{\mathbf{B},K}$ to simplify the notation. 

Now we are ready to recall the residue formula proved in \cite{OTASz}.
\begin{theorem}\label{verlinde}
Let $ c\in\Delta $ be a regular element, which thus 
specifies a chamber in $ \Delta $ and a parabolic moduli 
space $ P_0(c) $. Then for a diagonal basis $ \DD $, 
an arbitrary element $\lambda\in\Lambda$, and a positive integer $ k $, 
the Euler characteristic of the line bundle $\L(k;\lambda)$ (cf. \S\ref{S2.1}) is equal to 
\begin{equation}\label{eqnpoly}
\chi(P_0(c),\L(k;\lambda))
	=  N_{r,k}\cdot\sum_{\bb\in\DD}
	\iber_{\bb} [ w_\Phi^{1-2g}(x/\kk)\exp\scp{\lala/\kk}{x}] (- 
	[c]_\bb),
\end{equation}
where ${N}_{r,k} = (-1)^{{r \choose 2}(g-1)} r(r(k+r)^{r-1})^{g-1}$, $\lala=\lambda+\rho$ and $\kk=k+r$.
\end{theorem}

We will need the following property of the operator $\iber_{\mathbf{B},Q}$. 

\begin{lemma}\label{trivialshift}
Let $Q=(k+r)K-\delta\varphi\in\mathcal{H}^{\Phi}$, then for any vector $w\in \Lambda$ and a function $f\in \fw$, which depends on $\delta$, we have 
\begin{multline}\label{trivialshifteq}
\frac{\partial }{\partial \delta}\big{|}_{\delta=0}\iber_{\mathbf{B},Q}   \left[ f(x) \right](a+w)=  \\
\frac{\partial }{\partial \delta}\big{|}_{\delta=0} \iber_{\mathbf{B},Q}   \left[ f(x) \exp((k+r)w)\right](a) -  \iber_{\mathbf{B},(k+r)K}   \left[ f(x)\big{|}_{\delta=0}\exp((k+r)w)\phi_{\check{w}}(x) \right](a).
\end{multline}
\end{lemma}
\begin{proof}
Note that 
\begin{multline*}
\iber_{\mathbf{B},Q}   \left[ f(x) \right](a+w)=\iber_{\mathbf{B},Q}   \left[ f(x) \exp(Q_{\check{w}})\right](a)  = \\
\iber_{\mathbf{B},Q}   \left[ f(x) \exp((k+r)w-\delta\phi_{\check{w}}(x)) \right](a);
\end{multline*}
then taking the derivative with respect to $\delta$ at zero, we obtain the result.
\end{proof}

We are now ready to give the formula for the Euler characteristic of other vector bundles on the moduli spaces. 
\begin{theorem}\label{main}
Let $\mathcal{K}$ be the  canonical class of the curve $C$,  $ \lambda \in\Lambda$, $ k\in\mathbb{Z}_{>0} $, $\nu=(\nu_1\geq\nu_2 ... \geq\nu_r)\in\mathbb{Z}^r$, $\lala=\lambda+\rho$, $v_{\mathrm{det}}= (1,...,1,1-r)\frac{\sum\nu_i}{r}$,  $Q=(k+r)K-\delta\varphi^\nu \in \mathcal{H}^{\Phi}$
 and let ${c}\in\Delta$ be a regular element (cf. page \pageref{regular}).
Then for any diagonal basis $ 
	\DD \in \Bases $, the following equality holds:
\begin{multline}\label{eqmain}
	\chi(P_0({c}),\L({k;\lambda})\otimes\pi_!({U_\nu}\otimes\mathcal{K}^{\frac{1}{2}}))
	= \\  N_r \cdot  \frac{\partial }{\partial \delta}\big{|}_{\delta=0} \sum_{\bb\in\DD}
	\iber_{\bb, Q}\left[\mathrm{Hess}(Q(x))^{g-1}w_\Phi^{1-2g}(x)\exp(\scp{\lala+v_{\mathrm{det}}}{x})\right] \left( -[{c}]_{\bb}     \right),
\end{multline}
where $N_r= (-1)^{{r \choose 2}(g-1)} r^g$.
\end{theorem}
Taking the derivative with respect to $\delta$, we obtain the following explicit formulas.
\begin{corollary}\label{corrmain} Let $ \lambda, k, \nu,  Q, {c}$ and ${N}_{r,k} $ be as above, then 
\begin{multline*}
	\chi(P_0({c}),\L({k;\lambda})\otimes\pi_!({U_\nu}\otimes\mathcal{K}^{\frac{1}{2}})) = \\
	{N}_{r,k}\sum_{\bb\in\DD}
	\iber_{\bb}
	\bigg{[}w_\Phi^{1-2g}(x/\kk)\exp(\scp{\lala+v_{\mathrm{det}}}{x/\kk}) \bigg{(}\frac{-g}{k+r}\mathrm{tr}(\mathrm{Hess}(\varphi^\nu(x/\kk)))- \\
	\sum_i \frac{{\varphi}^\nu_{\check{\beta}^{[i]}}(x/\kk)\exp (\scp{\beta^{[i]}}{x})}{1-\exp (\scp{\beta^{[i]}}{x})}  
	+ \sum_i \scp{[{c}]}{\check{\beta}^{[i]}}{\varphi}^\nu_{\check{\beta}^{[i]}}(x/\kk)\bigg{)}   \bigg{]} \left( 
	-[{c}]_{\bb}\right).
\end{multline*}	

\end{corollary}
\begin{remark}\label{remres}
As explained in \cite[Remark 4.3]{OTASz}, the operator $\iber_{\bb}$  may be written as an \textit{iterated residue}: for $i=1,...,r-1$ we define $y_i=\scp{{\beta}^{[i]}}{x}$ and write $f$ and $a$ in these coordinates: $f(x)=\hat{f}(y)$, $\scp{a}{x}=\scp{\hat{a}}{y}$. Then 
\[  \iber_{\mathbf{B}}   \left[ f(x) \right] (a) =
\res_{y_1=0}\dots\res_{y_{r-1}=0}\frac{ 
	\hat f(y)\exp \scp {\hat a}y\;dy_1
	\wedge\dots\wedge dy_{r-1}}{(1-exp(y_1))\;
	\dots(1-exp(y_{r-1}))\;},
\]
where \textit{iterating} the residues means that at each step  we keep the variables with lower indices as unknown constants.
\end{remark}
\begin{example}\label{ex:2form}
Denote by $U$ the normalized universal bundle on  the moduli spaces of rank-3 parabolic bundles $P_0(>)$ and $P_0(<)$ defined in Example \ref{ex:wall}. We have $U\simeq U_\nu$ for $\nu=(1,0,0)$ and 
\begin{multline*}
\phi(\alpha^{12},\alpha^{23})=e^{\frac{2\alpha^{12}+\alpha^{23}}{3}}+e^{\frac{\alpha^{23}-\alpha^{12}}{3}}+e^{\frac{-\alpha^{12}-2\alpha^{23}}{3}}; \,\,\,
\phi_{\check{\alpha}^{12}}(\alpha^{12},\alpha^{23})=e^{\frac{2\alpha^{12}+\alpha^{23}}{3}}-e^{\frac{\alpha^{23}-\alpha^{12}}{3}};  \\
\phi_{\check{\alpha}^{23}}(\alpha^{12},\alpha^{23})=e^{\frac{\alpha^{23}-\alpha^{12}}{3}}-e^{\frac{-\alpha^{12}-2\alpha^{23}}{3}}; \,\,\,
\mathrm{tr}(\mathrm{Hess}(\phi(\alpha^{12},\alpha^{23}))=\frac{2}{3}\phi(\alpha^{12},\alpha^{23}).
\end{multline*}

Let $\DD$ be the 
diagonal basis from Example \ref{ex:diag}; writing the operator $\iber_{\bb}$  for $\bb\in\DD$  in the variables $(x,y)$ as explained in Remark \ref{remres} and using \cite[Remark 4.6]{OTASz}, we obtain
\begin{multline*}
 \chi(P_0(<),\mathcal{L}(k;\lambda)\otimes\pi_!(U\otimes\mathcal{K}^{\frac{1}{2}})= \\  N \cdot \res_{y=0}\res_{x=0}  \frac{(e^{\lambda_1x+(\lambda_1+\lambda_2)y+x+y+\frac{x+2y}{3}}-e^{\lambda_1x+(\lambda_1+\lambda_3)y+x+\frac{x-y}{3}})}{(1-e^{x(k+3)})(1-e^{y(k+3)})w_\Phi(x,y)^{2g-1}}\, \cdot \\ \left( \frac{2g}{3(k+3)}\phi(x,y)+ \frac{e^{(k+3)x}\phi_{\check{x}}(x,y)}{(1-e^{(k+3)x})}+
 \frac{e^{(k+3)y}\phi_{\check{y}}(x,y)}{(1-e^{(k+3)y})} \right)dx dy
\end{multline*}
and 
\begin{multline*}
 \chi(P_0(>),\mathcal{L}(k;\lambda)\otimes\pi_!(U\otimes\mathcal{K}^{\frac{1}{2}})= \\  N\cdot \res_{y=0}\res_{x=0} 
 \frac{(e^{\lambda_1x+(\lambda_1+\lambda_2)y+x+y+\frac{x+2y}{3}}-e^{\lambda_1x+(\lambda_1+\lambda_3)y+x+(k+3)y+\frac{x-y}{3}})}{(1-e^{x(k+3)})(1-e^{y(k+3)})w_\Phi(x,y)^{2g-1}}\, \cdot  \\ \left( \frac{2g}{3(k+3)}\phi(x,y)+ \frac{e^{(k+3)x}\phi_{\check{x}}(x,y)}{(1-e^{(k+3)x})}+
 \frac{e^{(k+3)y}\phi_{\check{y}}(x,y)}{(1-e^{(k+3)y})}  \right)dx dy - \\ 
  N\cdot \res_{y=0}\res_{x=0} 
 \frac{e^{\lambda_1x+(\lambda_1+\lambda_3)y+x+(k+3)y+\frac{x-y}{3}}\phi_{\check{y}}(x,y) }{(1-e^{x(k+3)})(1-e^{y(k+3)})w_\Phi(x,y)^{2g-1}} dx dy , 
\end{multline*}
where 
$w_\Phi(x,y)=2\mathrm{sinh}(\frac{x}{2})2\mathrm{sinh}(\frac{y}{2})2\mathrm{sinh}(\frac{x+y}{2})$ and $N=(-1)^{g}(3(k+3)^2)^{g}$.
One can compare these formulas with the ones from \cite[Example 4]{OTASz}.
\end{example}
\subsection{Wall-crossing in residue formulas}\label{S3.2}
We start the proof of Theorem \ref{main} following the strategy described in \S\ref{intro}. 
Our first step is to calculate the wall-crossing terms of the residue expressions from  Theorem \ref{main}. 
We choose two regular elements $c^+, c^-\in \Delta$ in two neighbouring chambers separated by the wall $S_{\Pi, l}$ (cf. \eqref{Swall}) such that 
$$[c^+_{\Pi'}]=l \text{\, and \,} [c^-_{\Pi'}]=l-1,$$ where we use the notation $c_{\Pi'}=\sum_{i\in\Pi'}c_i$ for $c\in\Delta$.  We denote by 
$$R_{\pm}^\nu(k,\lambda) =N_r\cdot \frac{\partial }{\partial \delta}\big{|}_{\delta=0} \sum_{\bb\in\DD}
	\iber_{\bb, Q}\left[\mathrm{Hess}(Q(x))^{g}w_\Phi^{1-2g}(x)\exp(\scp{\lala+v_{\mathrm{det}}}{x})\right] \left( -[{c^\pm}]_{\bb}     \right) $$
 the two polynomials in $(k,\lambda)\in\mathbb{Z}_{>0}\times\Lambda$. Then the wall-crossing term in the residue formula is the difference 
 $$R_{+}^\nu(k,\lambda)-R_{-}^\nu(k,\lambda).$$
 Using  \cite[Lemma 4.11]{OTASz}, we obtain the following expression for this difference.
 
\begin{lemma}\label{wcresidue}
 Let $(\Pi, l)$ and $c^+$, $c^-$ be as above, and fix a diagonal basis $\mathcal{D}\subset \Bases$. Denote by $ 
\DD|\Pi $	the subset of those elements $\bb$ of $ \DD $ for which $\tree(\bb)$ (cf. $\S\ref{S2.3}$) is a union of a tree on $\Pi'$, a tree on $\Pi''$ and a single edge ${\beta}_{\mathrm{link}}$ (which we will call the link) connecting $\Pi'$ and $\Pi''$.
Then 
\begin{multline*}
		R_{+}^\nu(k,\lambda)-R_{-}^\nu(k,\lambda)  =N_r\cdot \frac{\partial }{\partial \delta}\big{|}_{\delta=0} \\ \sum_{\bb\in\DD|\Pi }
	\iber_{\bb, Q}\left[(1-\exp(Q_{\check{\beta}_{\mathrm{link}}}(x)))\mathrm{Hess}(Q)^{g-1}w_\Phi^{1-2g}(x)\exp(\scp{\lala+v_{\mathrm{det}}}{x})\right] \left( -[c^+]_{\bb}     \right) 
\end{multline*}
\end{lemma}
\begin{remark}
Note that the multiplication by $ 
(1-\exp(Q_{\check{\beta}_{\mathrm{link}}}(x))) $ in Lemma \ref{wcresidue} has the effect of 
canceling one of the factors in the denominator in the 
definition \eqref{defiber} of the operation $ \iber $.
\end{remark}
As observed in \cite{OTASz}, even though this difference does not depend on the choice of $\mathcal{D}$, it is convenient to choose a particular diagonal basis (cf. \cite[page 19]{OTASz}). Introducing the notation $\rr'$ and $\rr''$ \label{Fi'Fi''} for the $A_{r'}$ and $A_{r''}$ root systems corresponding to $\Pi'$ and $\Pi''$, using \cite[Lemma 4.15]{OTASz} and 
taking the derivative with respect to $\delta$ at $\delta=0$, we arrive at the following statement. 
\begin{corollary}\label{corwcresidue} Let $ \DD' $ and $ \DD'' $ be 
	diagonal bases of $ \Phi^\p $ and $ \Phi^{\pp}$ 
	correspondingly.  Then 
	\begin{multline}\label{wcfinal}
		R_{+}^\nu(k,\lambda)-R_{-}^\nu(k,\lambda)= 
(k+r){N}_{r,k}\sum_{\bb^\p\in\DD^\p} \sum_{\bb^{\pp}\in\DD^{\pp}} \res_{\alphalink=0}\iber_{\bb^\p}\iber_{\bb^{\pp}} \\		
	\bigg{[}w_\Phi^{1-2g}(x/\kk)\exp(\scp{\lala+v_{\mathrm{det}}}{x/\kk}) \bigg{(}  \frac{-g}{k+r}\mathrm{tr}(\mathrm{Hess}(\varphi^\nu(x/\kk))) + l{\varphi}^\nu_{\check{\beta}_{\mathrm{link}}}(x/\kk) - \\ \sum_{i\neq\mathrm{link}} \frac{{\varphi}^\nu_{\check{\beta}^{[i]}}(x/\kk)\exp(\scp{\beta^{[i]}}{x})}{1-\exp(\scp{\beta^{[i]}}{x})}  + 
	\sum_{i\neq \mathrm{link}} \scp{[{c^+}]}{\check{\beta}^{[i]}}{\varphi}^\nu_{\check{\beta}^{[i]}}(x/\kk)  \bigg{)} \bigg{]} \left( 
	-[{c^+}]_{\bb}\right)\,d\alphalink,
	\end{multline}
where $ \res_{\alphalink=0}	
\iber_{B^\p}\iber_{B^{\pp}}\,d\alphalink $ is simply $ \iber_{\bb} $ (cf 
\eqref{defiber}) with $ \bb $ obtained by appending $\bb'$, 
and then $ \bb^{\pp} $ to $ \alphalink$, 
and  the factor $ (1-\exp\scp{\alphalink}x) $ 
removed from the denominator.
\end{corollary}

\begin{example}\label{ex:polydiff}
Calculating the difference of the two polynomials from Example \ref{ex:2form}, we obtain the wall-crossing term:
\begin{multline*}
- N \cdot \res_{y=0}\res_{x=0}  \frac{e^{\lambda_1x+(\lambda_1+\lambda_3)y+x+\frac{x-y}{3}}}{(1-e^{x(k+3)})w_\Phi(x,y)^{2g-1}}  \left( \frac{2g}{3(k+3)}\phi(x,y)+ \frac{e^{(k+3)x}\phi_{\check{x}}(x,y)}{(1-e^{(k+3)x})} \right)dx dy.
 \end{multline*}

\end{example}

\end{section}
\begin{section}{Wall-crossing in Euler characteristics}\label{S4}

In this section, we calculate the changes in Euler characteristics of vector bundles when varying the moduli spaces of parabolic bundles. The main result is Proposition \ref{wcrprop}, where we present explicit formulas for the wall-crossing terms for the left-hand side of  \eqref{eqmain}. 

\subsection{Wall-crossing in master space}\label{S4.1}
Fix the wall $S_{\Pi,l}$ given by an ordered partition $\Pi = (\Pi', \Pi'')$ of the first $r$ integers and an integer $l$, and two regular elements $c^+, c^-\in\Delta$ in two neighbouring chambers separated by the wall $S_{\Pi, l}$. 
Let $$c'=\sum_{i\in\Pi'} x_i   \text{\, and \,} c''=\sum_{i\in\Pi''} x_i.$$
Following Thaddeus  \cite{ThaddMast},
 one can construct the "master space" $Z$ whose quotients, under different linearizations, by a fixed $\mathbb{C}^*$-action, are the moduli spaces of $c^\pm$-stable parabolic bundles.
In \cite[\S5]{OTASz}, we showed that the elements $c^\pm$ may be chosen within their chambers so that $Z$ is a smooth, projective variety with a $\mathbb{C}^*$-action, and  identified the connected components of the fixed locus:
$$Z^{\mathbb{C}^*} \simeq P_0(c^+)\sqcup P_0(c^-)\sqcup Z^0,$$
where $Z^0$ is the set of points representing rank-r vector bundles $W$ on $C$, such that $W$ splits as a direct sum $W^\p\oplus{W^{\pp}}$, where $W^\p$ and $W^{\pp}$ are, respectively, $c^\p$ and $c^{\pp}$-stable parabolic bundles of degree $l$ and $-l$, rank $r'=|\Pi'|$ and $r''=|\Pi''|$:
$$Z^0 = \{W=W^\p\oplus W^{\pp} \, | \,W^\p\in \widetilde{P}_{l}(c^\p); \, W^{\pp}\in\widetilde{P}_{-l}(c^{\pp}); \, \, det(W)\simeq\mathcal{O}\}.$$

\begin{remark}\label{CohSplit}
Note that $Z^0$ is fibered over $\mathrm{Jac}^l$ with fibre $P_l(c^\p)\times{P_{-l}(c^{\pp})}$ by the determinant map $\widetilde{P}_l(c^\p)\to \mathrm{Jac}^l$, and 
$$H^*(Z^0,\mathbb{Q})\simeq H^*(P_l(c^\p)\times{P_{-l}(c^{\pp})},\mathbb{Q})\otimes H^*(\mathrm{Jac}^l,\mathbb{Q}). \qed$$
\end{remark}

Consider the polynomials 
 $$\chi^\nu_\pm(k; \lambda) = \chi(P_0(c^\pm), \mathcal{L}(k;\lambda)\otimes\pi_!({U_\nu}\otimes{\mathcal{K}}^{\frac{1}{2}})).$$
Our goal is to calculate the difference $\chi^\nu_+(k; \lambda)-\chi^\nu_-(k; \lambda)$.
 
 Applying the Atiyah-Bott fixed-point formula to the master space $Z$ with the $\mathbb{C}^*$-action, we showed \cite[Theorem 5.6]{OTASz} that  the wall-crossing polynomial $\chi^\nu_+(k; \lambda)-\chi^\nu_-(k; \lambda)$ is equal to 
 \begin{equation}\label{integralt}
\begin{split}
\res_{u=0}\int_{Z^0} \frac{ch((\mathcal{L}(k;{\lambda})\otimes\pi_!({U_\nu}\otimes{\mathcal{K}}^{\frac{1}{2}}))\big{|}_{Z^0})}{E(N_{Z^0})}\mathrm{Todd}(Z^0) \, du,
\end{split}
\end{equation}
where $E(N_{Z^0})$ is the K-theoretical Euler class of the conormal bundle of $Z^0$ in $Z$ and  $u$ is an equivariant parameter.

Before we calculate this integral, we need to introduce some extra notations. 

\subsection{Restriction. Representations}
For any weight $\nu = (\nu_1,...,\nu_r)$ of $GL_r$, we define  $$|{\nu}|\overset{\mathrm{def}}={\sum_i\nu_i};$$   the irreducible representation $\rho_\nu$  of $GL_r\simeq (SL_r\times \mathbb{C}^*)/\mathbb{Z}_r$ with highest weight $\nu$ can be decomposed by restriction as a product of the irreducible representation $\bar{\rho}_\nu$ of $SU_r$ and the one-dimensional representation ${\rho}[|\nu|]: t\mapsto t^{|\nu|}$  of the center  $Z(GL_r)\simeq\mathbb{C}^*$.

Let $GL_{r'}\times GL_{r''}$ be the subgroup of $GL_r$ induced by an ordered  partition $(\Pi',\Pi'')$ of the first $r$ positive integers. The restriction of the irreducible representation $\rho_\nu$ of $GL_r$ decomposes as a direct sum of irreducible representations of $GL_{r'}\times GL_{r''}$:
$$\rho_\nu=\sum_{(\nu', \nu'')}\rho_{\nu'}\otimes\rho_{\nu''}.$$ 
Similarly, the restriction of the representation $\bar{\rho}_\nu$ to $SU_r\cap (GL_{r'}\times GL_{r''})\subset GL_r$  can be decomposed as  a direct sum 
$$\bar{\rho}_\nu=\sum_{(\nu', \nu'')}\bar{\rho}_{\nu'}\otimes\bar{\rho}_{\nu''}\otimes \rho[{{rs}}]\,\,$$ of products of irreducible representations of $SU_{r'}$, $SU_{r''}$ and  the one-dimensional  torus $\mathbb{C}^*\simeq (Z(GL_{r'})\times Z(GL_{r''}))\cap SU_r,$ where $s= \sum_{i\in\Pi'}(\nu'_i-|\nu|/r).$ Let $$w\overset{\mathrm{def}}=\frac{s}{r'}\sum_{i\in\Pi'}{x_i}-\frac{s}{r''}\sum_{i\in\Pi''}{x_i}\in V^*,$$
then the corresponding  decomposition of character functions (cf. end of \S\ref{S2.11}) has the form 
\begin{equation}\label{phirestr}
\phi^\nu= \sum_{(\nu', \nu'')}\phi^{\nu'}\phi^{\nu''}\exp(w).
\end{equation}
Recall (cf. Lemma \ref{wcresidue}) that the a key role in our wall-crossing terms is played by the bases $\bb$ of $V^*$, obtained by appending $\bb'$, 
and then $ \bb^{\pp} $ to $ \alphalink$. 
Using expression \eqref{phirestr}, we arrive at the following equalities for the directional derivatives of $\phi^\nu$. 
\begin{lemma}\label{phideriv}
Let $\bb = \alphalink \,\, \bb'\,\,\bb'' $  be a basis of $V^*$ described above.  Then:
\begin{enumerate}[(i)]
\item For any $\alpha\in \bb'$ and any $\beta\in \bb''$  we have
$$\phi^\nu_{\check{\alpha}}=\sum_{(\nu', \nu'')}\phi^{\nu'}_{\check{\alpha}}\phi^{\nu''}\exp(w) \text{\,\,\, \,\, and \,\, \,\,\,} \phi^\nu_{\check{\beta}}=\sum_{(\nu', \nu'')}\phi^{\nu''}_{\check{\beta}}\phi^{\nu'}\exp(w);$$
\item $\check{\beta}_{\mathrm{link}}= \frac{r''}{r}\sum_{i\in\Pi'}{x_i}-\frac{r'}{r}\sum_{i\in\Pi''}{x_i}$, and thus 
$$\phi^\nu_{\check{\beta}_{\mathrm{link}}}=\sum_{(\nu', \nu'')}\frac{sr}{r'r''}\phi^{\nu'}\phi^{\nu''}\exp(w);$$
\item $\mathrm{tr}(\mathrm{Hess}(\phi^\nu))=$ $$\sum_{(\nu', \nu'')} (\mathrm{tr}(\mathrm{Hess}(\phi^{\nu'}))\phi^{\nu''}+\mathrm{tr}(\mathrm{Hess}(\phi^{\nu''}))\phi^{\nu'}+
 s^2\left(\frac{r}{r'r''}\right)\phi^{\nu'}\phi^{\nu''})\exp(w).$$
\end{enumerate}
\end{lemma}

\subsection{Restriction. Bundles}\label{S4.3}
Recall that our goal is to calculate the integral \eqref{integralt}; our first step is to identify the characteristic classes under this integral.
We showed in \cite[Theorem 6.11]{OTASz}, that 
\begin{multline}\label{restrline}
\res_{u=0}\int_{Z^0} \frac{ch(\mathcal{L}(k;{\lambda})\big{|}_{Z^0})}{E(N_{Z^0})}\mathrm{Todd}(Z^0) \, du =  (k+r){N}_{r,k}
			\sum_{\textbf{B}^\p\in\mathcal{D}'}\sum_{\textbf{B}^{\pp}\in\mathcal{D}^{\pp}}\\
			\res_{{\beta}_{\mathrm{link}}
			=0}\iber_{\textbf{B}^\p}\iber_{\textbf{B}^{\pp}}[w_\Phi(x/\widehat{k})^{1-2g}\exp(\scp{\lala}{x/\kk})](-[c^+]_{\textbf{B}})
			 \, d{\beta}_{\mathrm{link}},
\end{multline}
where $\phi\in\Sigma_r$ is the unique permutation which sends $\{1,...,r'\}$ to $\Pi'$ preserving the order of the first $r'$ and the last $r''$ elements. Now we study the restriction of the class $ch(\pi_!(U_\nu\otimes\mathcal{K}^\frac{1}{2}))$ to $Z^0$.

Let $\omega\in H^2(C)$ be the fundamental class of our curve $C$,  and $e_1, 
..., e_{2g} $ a basis of $H^1(C)$, such that $e_ie_{i+g}=\omega$ for $1\leq 
i\leq g$, and all other intersection numbers $e_ie_j$ equal 0. For a class 
$\gamma\in H^*(P\times C)$ of a product, we introduce the following notation 
for its K\"unneth components (cf. 
\cite{Wittenrevisited}):
\begin{equation}\label{Kunneth}
\gamma=\gamma_{(0)}\otimes1+\sum_i\gamma_{(e_i)}\otimes 
e_i+\gamma_{(2)}\otimes\omega\in \bigoplus_{i=0}^2H^{*-i}(P)\otimes H^i(C).
\end{equation}
It follows from the Groethendieck-Riemann-Roch theorem that 
\begin{equation}\label{part2}
ch(\pi_!(U_\nu\otimes\mathcal{K}^\frac{1}{2}))=ch(U_\nu)_{(2)}.
\end{equation}

Recall our notation $\mathcal{J}$ for the Poincare bundle over $\mathrm{Jac}\times C$, 
such that  $c_1(\mathcal{J})_{(0)}=0$ and an element $\eta\in H^2(\mathrm{Jac})$ defined by $(\sum_ic_1(\mathcal{J})_{(e_i)}\otimes e_i)^2=-2\eta\otimes\omega$; then (cf. \cite{Zagier}) for any $m\in\mathbb{Z}$
\begin{equation}\label{zagjac}
\int_{\mathrm{Jac}}e^{\eta m}=m^g.
\end{equation}
The following statement is straightforward. 
\begin{lemma}\label{chUnuZ} Denote by $U[l]'$ and  by  $U[-l]''$ the normalized universal bundles on the moduli spaces $P_l(c')\times C$ and  $P_{-l}(c'')\times C$, correspondingly (cf. beginning of \S\ref{S2.11}). 
Let $U[l]_{\nu'}$ and $U[-l]_{\nu''}$ be the the associated vector bundles on $P_l(c')\times C$ and $P_{-l}(c'')\times C$ (cf. 
 \eqref{constructionUnu}). Then 
\begin{multline*}
ch(U_\nu\big{|}_{Z^0})=\sum_{(\nu', \nu'')}\exp(u\sum_i\nu'_i)\Big{(}ch(U[l]_{\nu'})\boxtimes ch(U[-l]_{\nu''})\boxtimes \\ \left(1+\left(l\frac{sr}{r'r''}-\left(\frac{sr}{r'r''}\right)^2\eta\right)\otimes\omega\right)\Big{)}.
\end{multline*}
\end{lemma}
Putting Lemma \ref{chUnuZ} and equation \eqref{part2} together, we obtain the following statement.
\begin{lemma}\label{chparts} In the notation of Lemma \ref{chUnuZ}, 
\begin{multline*}
ch(\pi_!(U_\nu\otimes\mathcal{K}^\frac{1}{2})\big{|}_{Z^0})= \sum_{(\nu', \nu'')} \exp(u\sum_i\nu_i')\Big{(}ch(U[l]_{\nu'})_{(2)}\boxtimes ch(U[-l]_{\nu''})_{(0)}+ \\ ch(U[l]_{\nu'})_{(0)}\boxtimes ch(U[-l]_{\nu''})_{(2)}+ ch(U[l]_{\nu'})_{(0)}\boxtimes ch(U[-l]_{\nu''})_{(0)}\boxtimes \left(l\frac{sr}{r'r''}-\left(\frac{sr}{r'r''}\right)^2\eta\right) \Big{)}.
\end{multline*}
\end{lemma}

\begin{example}\label{ex:ch2}
It follows from Example \ref{ex:wall} that in rank 3 case $\Pi'=\{2\}$,  $\Pi''=\{1,3\}$ and  the fixed locus $Z^0$ is the set of vector bundles that split as a direct sum of rank-2 degree-0 stable parabolic bundle and a line bundle of degree 0.  We denote by $U''$ the normalized universal bundle on the moduli space $P_0$ of rank-2 stable parabolic bundles with trivial determinant. Then for the universal bundle $U$ from Example \ref{ex:2form}, the Chern character $ch(U\big{|}_{Z^0})$ has two summands:
\begin{itemize}
\item for $\nu' = (0), \nu''=(1,0)$ we have $ch(U'')_{(2)}-\frac{1}{4}\, \eta \,ch(U'')_{(0)}$;
\item for $\nu' = (1), \nu''=(0,0)$ we have $e^u\, \eta $.
\end{itemize}
\end{example}

\begin{remark}\label{ch0}
Recall that in \cite[\S6.2]{OTASz} we identified the functions on $V$ with cohomology classes on $P_0(c)$ and an equivariant cohomology classes on $Z^0$. Under these identifications, $\ch(U_{\nu'})_{(0)}$ corresponds to the function $\phi^{\nu'}(x)\exp\scp{v_{\mathrm{det}}'}{x}$ and $\ch(U_{\nu''})_{(0)}$ corresponds to the function $\phi^{\nu''}(x)\exp\scp{v_{\mathrm{det}}''}{x}$.
\end{remark}

Now our goal is to calculate the wall-crossing integral \eqref{integralt} applying induction by rank based on Theorem \ref{main}. 
Using \eqref{part2}, we can write the inductive hypothesis in the following form:
\begin{multline}\label{Corpoly1}
	\int\limits_{P_0(c)}
	ch(\L(k;\lambda))ch(U_\nu)_{(2)}\mathrm{Todd}(P_0(c))
	= \\ 
	N_r\cdot \frac{\partial }{\partial \delta}\big{|}_{\delta=0} \sum_{\bb\in\DD}
	\iber_{\bb, Q}\left[\mathrm{Hess}(Q(x))^{g-1}w_\Phi^{1-2g}(x)\exp(\scp{\lala+v_{\mathrm{det}}}{x})\right] \left( -[{c}]_{\bb}     \right).
\end{multline}
Fixing $k$ and varying $\lambda$, we can extend this hypothesis by linearity to the following linear combinations of Chern characters of line bundles  $$ \sum_i ch(\L(k;\lambda^i)) = ch(\L(k;0))\cdot\sum_ich(\L(0;\lambda^i)). $$ Since any polynomial on 
$ V $, up to a fixed degree may be represented as a linear 
combination of exponential functions of the form $ 
\exp\scp\lambda{x} $, formula \eqref{Corpoly1} may be 
generalized in the following way.

\begin{lemma}\label{Wint} Let $ G(x) $ be a formal power series on $ V 
$, and denote by $ G(z) $ the characteristic class in $ 
H^*(P_0(c)) $ obtained by the identification of functions on $V$ and cohomology classes on $P_0(c)$ (cf. Remark \ref{ch0}). Then 
	\begin{multline}\label{Corpoly2}
		\int\limits_{P_0(c)}
		ch(\L_0(k;0))G(z)ch(U_\nu)_{(2)}\mathrm{Todd}(P_0(c))
		= N_r\cdot \frac{\partial }{\partial \delta}\big{|}_{\delta=0} \\ \sum_{\bb\in\DD}
		\iber_{\bb, Q}\left[\mathrm{Hess}(Q(x))^{g-1}G(x)w_\Phi^{1-2g}(x)\exp(\scp{\lala+v_{\mathrm{det}}}{x})\right] \left( -[{c}]_{\bb}     \right).
	\end{multline}	
\end{lemma}

Armed with this statement and equality \eqref{restrline}, we are ready to calculate the integral \eqref{integralt}. We start with the case $l=0$. 
\begin{itemize}
\item Note that for $l=0$, $[c^+]=[c']+[c'']$.
\item Then using the induction hypothesis  \eqref{Corpoly2} and Remark \ref{ch0}, we conclude that the  first summand in Lemma \ref{chparts} contributes 
\begin{multline}\label{onepart}
	(k+r){N}_{r,k} \sum_{(\nu', \nu'')} \res_{u=0} \exp(u\sum\nu_i')\sum_{\bb'\in\DD'}\sum_{\bb''\in\DD''} \iber_{\bb^\p}\iber_{\bb^{\pp}} 
	\Big{[}w_\Phi^{1-2g}(x/\kk) \\ \exp(\scp{\lala+v_{\mathrm{det}}'+v_{\mathrm{det}}''}{x/\kk}) \bigg( \frac{-g}{k+r}\mathrm{tr}(\mathrm{Hess}(\varphi^{\nu'}(x/\kk)))  \varphi^{\nu''}(x/\kk)  - \sum_{\beta^{[i]}\in\bb'} \varphi^{\nu''}(x/\kk)\cdot \\ 
	 \frac{{\varphi}^{\nu'}_{\check{\beta}^{[i]}}(x/\kk)\exp (\scp{\beta^{[i]}}{x})}{1-\exp (\scp{\beta^{[i]}}{x})}  + 
	\sum_{\beta^{[i]}\in\bb'} \scp{[{\k c}]}{\check{\beta}^{[i]}}{\varphi}^{\nu'}_{\check{\beta}^{[i]}}(x/\kk) \varphi^{\nu''}(x/\kk) \bigg) \Big{]} \left(-[c^+]_{\bb}\right) \, du
\end{multline}
to the wall-crossing integral \eqref{integralt}.
\item Note that 
\begin{equation}\label{sumdet}
\scp{v_{\mathrm{det}}'+v_{\mathrm{det}}''}{x}+(x_{\phi(r')}-x_r)\sum\nu_i = v_{\mathrm{det}}+w,
\end{equation}
 hence after the identification of $u$ with $\beta_{\mathrm{link}}= x_{\phi(r')}-x_r$ justified in \cite[equation (33)]{OTASz} (see also Remark \ref{ch0}), we can replace the factors 
$$\exp(\scp{v_{\mathrm{det}}'+v_{\mathrm{det}}''}{x/\kk})\exp(u\sum\nu_i') =  \exp(\scp{v_{\mathrm{det}}+w}{x/\kk})$$
 in \eqref{onepart}.
\item The second summand in Lemma \ref{chparts} has the same form as \eqref{onepart} with exchanged $\nu'$ and $\nu''$, $\bb'$ and $\bb''$.
\item Since $$\int_{\mathrm{Jac}}\left(\exp(\eta\frac{(k+r)r}{r'r''})-\left(\frac{sr}{r'r''}\right)^2\eta\right)=\frac{-g}{(k+r)}\frac{s^2r}{r'r''}\left(\frac{(k+r)r}{r'r''}\right)^g,$$ where the first summand comes from the restriction of $\mathcal{L}(k;\lambda)/E(N_{Z^0})$ to $Z^0$ (cf. \cite[Lemma 6.4, Proposition 6.7]{OTASz}),
the third summand in Lemma \ref{chparts} for $l=0$ contributes  
\begin{multline*}
-g\,{N}_{r,k}\frac{s^2r}{r'r''} \sum_{(\nu', \nu'')} \res_{\beta_{\mathrm{link}}=0}\sum_{\bb'\in\DD'}\sum_{\bb''\in\DD''} \iber_{\bb^\p}\iber_{\bb^{\pp}}[w_\Phi^{1-2g}(x/\kk)\exp(\scp{\lala}{x/\kk}) \\ \exp(\scp{v_{\mathrm{det}}+w}{x/\kk}) \varphi^{\nu'}(x/\kk) \varphi^{\nu''}(x/\kk)]\left(-[c^+]_{\bb}\right) \, d \beta_{\mathrm{link}}.
\end{multline*}
to the wall-crossing integral \eqref{integralt}.
\item Finally,  using Lemma\ref{phideriv}, we arrive at the following statement for $l=0$.
\end{itemize}
\begin{proposition}\label{wcrprop}
Let $ \DD' $ and $ \DD'' $ be 
	diagonal bases of $ \Phi^\p $ and $ \Phi^{\pp}$ 
	 and let $\beta_{\mathrm{link}}$ be the link edge (cf. page \pageref{Fi'Fi''}).  Then 
	\begin{multline}\label{wcfinal}
		\chi^\nu_+(k; \lambda)-\chi^\nu_-(k; \lambda)= 
(k+r){N}_{r,k}\sum_{\bb^\p\in\DD^\p} \sum_{\bb^{\pp}\in\DD^{\pp}} \res_{\alphalink=0}\iber_{\bb^\p}\iber_{\bb^{\pp}} \\		
	\bigg{[}w_\Phi^{1-2g}(x/\kk)\exp(\scp{\lala+v_{\mathrm{det}}}{x/\kk}) \bigg{(}  \frac{-g}{k+r}\mathrm{tr}(\mathrm{Hess}(\varphi^\nu(x/\kk)))  + 
      l{\varphi}^\nu_{\check{\beta}_{\mathrm{link}}}(x/\kk) - \\ \sum_{i\neq\mathrm{link}} \frac{{\varphi}^\nu_{\check{\beta}^{[i]}}(x/\kk)\exp(\scp{\beta^{[i]}}{x})}{1-\exp(\scp{\beta^{[i]}}{x})}  + 
	\sum_{i\neq \mathrm{link}} \scp{[{c^+}]}{\check{\beta}^{[i]}}{\varphi}^\nu_{\check{\beta}^{[i]}}(x/\kk)  \bigg{)} \bigg{]} \left( 
	-[{c^+}]_{\bb}\right)\,d\alphalink.
	\end{multline}
\end{proposition}
\begin{remark}\label{identification}
Note that this wall-crossing term coincides with the one from Corollary \ref{corwcresidue}, and hence with the one from Lemma \ref{wcresidue}.
\end{remark}
\begin{example}
Let $z=c_1(\mathcal{F}_2^{\pp}/\mathcal{F}_1^{\pp}\otimes{\mathcal{F}_1^{\pp}}^*)\in H^2(P_0)$, where $\mathcal{F}''_i$ are flag bundles on $P_0$ (cf. Example \ref{ex:ch2}). In particular, we have $ch(U'')_{(0)}=e^z+1$. 

We saw in \cite[Example 6]{OTASz} that in rank-3 case  the Chern character of the restriction of the line bundle $\mathcal{L}(k;\lambda)$ multiplied by the inverse of the K-theoretical Euler class of the conormal bundle of $Z^0$ is equal to 
$$\exp\left(\frac{3(k+3)\eta}{2}\right)e^{\lambda_2u}ch\left(\mathcal{L}''(k+1;\lambda_1)\right)e^{\frac{z}{2}}\left(2\mathrm{sinh}(u/2)2\mathrm{sinh}(({z-u})/{2})\right)^{2g-1},$$ where $\L''(k;\lambda)$ is a line bundle $\L(k,(\lambda,-\lambda))$ on  $P_0$.
Using Example \ref{ex:ch2} and Theorem \ref{verlinde}, we conclude (cf. \eqref{integralt}) the the wall-crossing term 
\begin{equation}\label{ex:wcrterm}
\chi(P_0(<),\mathcal{L}_0(k,\lambda)\otimes\pi_!(U\otimes\mathcal{K}^{\frac{1}{2}}))-\chi(P_0(>),\mathcal{L}_0(k,\lambda)\otimes\pi_!(U\otimes\mathcal{K}^{\frac{1}{2}}))
\end{equation}
 is equal to 
\begin{multline*}
-\left(\frac{3(k+3)}{2}\right)^g \res_{u=0}{e^{\lambda_2u}}  \int\limits_{P_0} \frac{ch(\mathcal{L}''(k+1;\lambda_1)\otimes\pi_!(U''\otimes\mathcal{K}^{\frac{1}{2}}))e^{\frac{z}{2}}}{(2\mathrm{sinh}(\frac{u}{2})2\mathrm{sinh}(\frac{z-u}{2}))^{2g-1}} \mathrm{Todd}(P_0) du \\
-\frac{2g}{3(k+3)}N\cdot\res_{u=0}\res_{z=0}\frac{e^{\lambda_1z+\lambda_2u+z}(e^u+\frac{1+e^z}{4})}{\tilde{w}_\phi(z,u)^{2g-1}(1-e^{(k+3)z})}dzdu,
\end{multline*}
where  $\tilde{w}_\phi(z,u)=2\mathrm{sinh}(\frac{z-u}{2})2\mathrm{sinh}(\frac{u}{2})2\mathrm{sinh}(\frac{z}{2})$ and $N=(-1)^g(3(k+3)^2)^g$. 
This integral is the Euler charactersitic of a vector bundle on the moduli space of degree-$0$ rank-$2$ stable parabolic bundles, so we can calculate it using the induction by rank (cf. formula \eqref{rank2final}). A simple calculation shows that the wall-crossing term \eqref{ex:wcrterm} is equal to 
\begin{multline*}
 - \frac{2g}{3(k+3)}N \cdot \res_{u=0}\res_{z=0}  \frac{e^{\lambda_1x+\lambda_2u+z}(1+e^u+e^z)}{(1-e^{z(k+3)})\tilde{w}_\phi(z,u)^{2g-1}} dz du \\
  -N \cdot \res_{u=0}\res_{z=0}  \frac{e^{\lambda_1x+\lambda_2u+z+(k+3)z}(1-e^z)}{(1-e^{z(k+3)})^2\tilde{w}_\phi(z,u)^{2g-1}} dz du  .
 \end{multline*}
Note that this is exactly the same polynomial as in Example \ref{ex:polydiff} after changing $(z,u)$ to $(x,-y)$.
\end{example}

If $l\neq 0$, we will need one more step to calculate the wall-crossing term \eqref{integralt}, which uses the tautological Hecke correspondences.
\subsection{Hecke correspondence}\label{S4.4} 

 In  \cite[Section 7]{OTASz} we defined the tautological Hecke operators between the moduli spaces of parabolic bundles with different degrees and parabolic weights as follows:  given a vector bundle $W$ on $C$ with a full flag $F_*$ in the fibre $W_p$ at $p\in C$, we consider the associated sheaf of sections $\mathcal{W}$ and define the subsheaf $$\mathcal{W}[-1]=\{\gamma\in H^0(C, \mathcal{W}) \, | \, \gamma(p)\subset F_{r-1}\} \subset\mathcal{W}.$$ Then  $\mathcal{W}[-1]$ is locally free, and thus defines a vector bundle, which we denote by $W[-1]$. Considering the associated morphism of vector bundles $W[-1]\to W$, we defined the full flag $G_*$ in the fibre $W[-1]_p$ and denoted this operator by $\mathcal{H}: (W, F_*)\mapsto (W[-1], G_*)$. We proved that $\mathcal{H}$ induces  an isomorphism of the moduli spaces  $$\mathcal{H}: P_d(c_1,c_2,...,c_r)\simeq P_{d-1}(c_2,...,c_r,c_1-1).$$ 

Applying $\mathcal{H}$ to  the normalized universal bundle $U$ on the moduli space $P_0(c)\times C$ we obtain a short exact sequence for the corresponding sheaves of sections:
\begin{equation*}\label{SES} 0 \to \mathcal{U}[-1]\to \mathcal{U} \to  \mathcal{F}_r/\mathcal{F}_{r-1} \to 0. \end{equation*}
Considering the associated vector bundles, we arrive at the following equality
\begin{equation}\label{chU}
ch(U)=ch(U[-1])+\omega \cdot ch(\mathcal{L}(0;(1,0,..,0,-1))),
\end{equation}
where $\omega\in H^2(C)$ is the fundamental class of the curve (cf. the beginning of \S\ref{S4.3}).
\begin{remark}\label{normalU}
Note that under the Hecke isomorphism $\mathcal{H}$, the normalized (cf. \S\ref{S2.11}) universal bundle  $U$ on the moduli space $P_0(c_1,c_2,...,c_r)\times C$ corresponds to the universal bundle  $U[-1]$ on the moduli space $P_{-1}(c_2,...,c_r,c_1-1)\times C$    such that the line bundle $\mathcal{F}_2/\mathcal{F}_1$ is trivial. 
\end{remark}

Similarly, applying the Hecke operator $\mathcal{H}^l$ to the normalized universal bundle $U$, we obtain the universal bundle $U[-l]$ on  $P_{-l}(c_{l+1},...,c_r,c_1-1,...,c_l-1)\times C$.

\noindent\textbf{Notation:} \label{weightmult} Given an irreducible representation $\rho_\nu: GL_r\to GL(V_\nu)$  of highest weight $\nu$, we consider its weight decomposition $$V_\nu=\bigoplus_{\mu\in\mathbb{Z}^r} V[\mu],$$ where $V[\mu]$ is the weight space of the weight $\mu$, and we denote by $m_\mu = \mathrm{dim}(V[\mu])$. 

\begin{proposition}\label{HeckeChern}  
Let $U[-l]_\nu$ be the vector bundle on $P_{-l}(c)\times C$ associated to the irreducible representation $\rho_\nu$ of $GL_r$ with highest weight $\nu$ and the normalized universal bundle $U[-l]$ (cf. \S\ref{S2.11}).  Then 
$$ch(U_\nu)=ch(U[-l]_\nu)+\omega\sum_{\mu}m_\mu(\mu_1+...+\mu_l)ch(\mathcal{L}(0;(\mu_1,...,\mu_{r-1},\mu_{r}-|\nu|))),$$
 where $|\nu|=\sum_i\nu_i$ and the sum runs over the weights $\mu$ of n $\rho_\nu$ with highest weight $\nu$.
\end{proposition}
\begin{proof}
Given a rank-$r$ vector bundle $V$ on $P_0(c)\times C$ and a symmetric polynomial $f\in\mathbb{C}[y_1,...,y_r]^{\Sigma_r}$, denote by $f(V) \in H^*(P_0(c)\times C)$ the cohomology class obtained by evaluating $f$ at the Chern roots of $V$. The flag $\mathcal{F}_1\subset \mathcal{F}_2\subset...\subset\mathcal{F}_r=U_p$ defines the cohomology classes 
$$\xi_i=c_1(\mathcal{F}_{r-i+1}/\mathcal{F}_{r-i}\otimes\mathcal{F}_1^*)\in H^2(P_0(c)),$$ and thus 
we have
 $$ch(U_p) = e^{\xi_1}+...+e^{\xi_{r-1}}+1;$$ it follows from Remark \ref{flagsection} that the Chern character of an associated bundle $U_\nu$ is given by
$$
ch((U_\nu)_p) = \sum_{\mu}m_\mu\exp(\mu_1\xi_1+...+\mu_{r-1}\xi_{r-1}).
$$
We note that the cohomology class  $f(U_p)$ 
 in $ H^*(P_0(c)\times C)$ is well-defined for any (not necesserly symmetric) polynomial $f\in\mathbb{C}[y_1,...,y_r]$.

We introduce the notation $f_i(y_1,...,y_r)=\frac{1}{i!}(y_1^i+...+y_r^i)$; in particular, for any vector bundle $V$ on $P_0(c)\times C$,  we have 
$f_i(V)=ch_i(V)$. It follows from 
\eqref{chU} that 
$$f_i(U)=f_i(U[-1])+\omega \, \partial_{y_1}f_i(U_p),$$ and thus
\begin{multline}\label{chderprod}
f_i(U)f_j(U) = 
f_i(U[-1])f_j(U[-1])+ \\ \omega \, ( \partial_{y_1}f_i(U_p)f_j(U[-1])+ \partial_{y_1}f_j(U[-1]_p)f_i(U)) = \\
f_i(U[-1])f_j(U[-1]) + \omega \,  \partial_{y_1}(f_if_j)(U_p).
\end{multline}
For the last equality, we used the facts that $\omega \, ch(U)=\omega \, ch(U_p)$ and that according to \eqref{chU}, $ch(U_p)=ch(U[-1]_p)$.

Since any symmetric polynomial $f\in\mathbb{C}[y_1,...,y_r]^{\Sigma_r}$ may be written as a polynomial in $f_i$'s, \eqref{chderprod} implies that for any symmetric polynomial $f$ we have:
$$f(U)=f(U[-1])+\omega\,\partial_{y_1}f(U_p).$$
Let 
$$g_\nu(y_1,...,y_r)= \sum_{\mu}m_\mu\exp(\mu_1y_1+...+\mu_ry_r);$$ since $g_\nu(U)=ch(U_\nu)$, we have 
$$ch(U_\nu)=ch(U[-1]_\nu) + \omega\,\partial_{y_1}g_\nu(U_p),$$
and thus  
$$ch(U_\nu)=ch(U[-1]_\nu) + \omega\,\sum_\mu m_\mu \mu_1\exp(\mu_1\xi_1+...+\mu_{r-1}\xi_{r-1}).$$
Finally, note that  
$$ \exp(\mu_1\xi_1+...+\mu_{r-1}\xi_{r-1}) = ch(\mathcal{L}(0;(\mu_1,...,\mu_{r-1},\mu_{r}-|\nu|))),$$
hence we obtain the proof for $l=1$. Iterating this argument, we obtain the proof for the general case. 
\end{proof}

\subsection{Wall-crossing for $l \neq 0$} 
Recall that our goal is to calculate the wall-crossing integral \eqref{integralt} for non-zero $l$, or, more precisely, to prove Proposition \ref{wcrprop} for the case when $l\neq 0$. The treatment of this case follows the logic of \cite[\S7.2]{OTASz}, hence, in this section, we will only highlight the differences which arise in our, more general, situation. 
For simplicity, we assume that $l$ is positive (the other case is analogous). 
\begin{itemize}
\item We first apply the Hecke operators 
$\mathcal{H}^{l}$ and $\mathcal{H}^{-l}$ to the moduli spaces $P_l(c')$ and $P_{-l}(c'')$ to obtain
$$P_0'=P_0(c'_{l+1},...,c'_{r'},c'_1-1,...,c'_l-1)\simeq P_l(c') \text{\, and \,}$$ $$P_0^{\pp}=P_0(c^{\pp}_{r^{\pp}-l+1}+1,...,c^{\pp}_{r^{\pp}}+1,c^{\pp}_1,...,c^{\pp}_{r^{\pp}-l})\simeq P_{-l}(c^{\pp}).$$
\item Next,  applying the Hecke operator $\mathcal{H}^{l}\times\mathcal{H}^{-l}$ to the wall-crossing term  \eqref{integralt}, we recast it as an integral over the moduli spaces of degree-$0$ parabolic bundles $P_0'\times P_0''$, and thus we can calculate this integral using the induction by rank as in $\S\ref{S4.3}$. 
\item As in \cite[page 33]{OTASz}, to arrive at Proposition \ref{wcrprop} we will need to make additional transformations of the  formulas we obtained.
We perform this transformation by applying Lemma \ref{trivialshift} with $$w=\sum_{i=1}^l (x_{\phi(r'-l+i)}-x_{\phi(r'+i)})\in\Lambda,$$ where $\phi\in\Sigma_r$ is the permutation which sends  $\{1,...,r'\}$ to $\Pi'$ preserving the order of the first $r'$ and the last $r''$ elements. 

The first summand on the right-hand side of \eqref{trivialshifteq} coincides with the shift of $\lambda$ we treated in \cite[page 34]{OTASz}. An easy calculation  shows that the second summand on the right-hand side of \eqref{trivialshifteq} eliminates the changes  (cf. Proposition \ref{HeckeChern} and equation \eqref{part2}) of the Chern character of $\pi_!({U_\nu}\otimes{\mathcal{K}}^{\frac{1}{2}})\big{|}_{Z^0}$ under the Hecke transformations $\mathcal{H}^{l}$ and $\mathcal{H}^{-l}$.
\end{itemize}
This completes the proof of Proposition \ref{wcrprop} for arbitrary $l\in\mathbb{Z}$.

\end{section}
\begin{section}{Symmetry}
The main result of this section is Proposition \ref{geometricsymm}, where we prove certain symmetry for the Euler characteristics of our vector bundles on the moduli spaces of parabolic bundles. 
\subsection{Symmetries through Serre duality}\label{S5.1}
Denote by $N_{\pm1}$ the moduli spaces of rank-$r$ degree-$\pm1$  stable 
 vector bundles and by $UN^\pm$ the universal bundle over $N_{\pm1}\times C$, normalized in such a way that $\mathrm{det}(UN^-_p)\simeq \mathcal{L}_{-1}(r;(-1,...,-1))$ and $\mathrm{det}(UN^+_p)\simeq \mathcal{L}_{1}(-r;(-1,...,-1))$. 
 
 In \cite[Lemma 8.3]{OTASz} we identified the moduli spaces $P_{1}(>)$ and $P_{-1}(<)$, which are isomorphic to the flag bundles $$P_{1}(>)  \simeq \mathrm{Flag}({UN}^+_p) \overset{p}{\to} N_1 \text{ \,\,\, and \,\,\,}   P_{-1}(<) \simeq \mathrm{Flag}({UN}^-_p) \overset{p}\to N_{-1}.$$
 The following is easy to verify. 
\begin{lemma}\label{O(1)}
Under the normalization described above, the line bundles $\mathcal{F}_1\subset p^*(UN^{\pm}_p)$ are isomorphic to 
 $\mathcal{L}_{-1}(-1;(0,...,0,1))$ and $\mathcal{L}_{1}(1;(0,...,0,1))$, respectively (cf. \S\ref{S2.11}).
\end{lemma}

Applying the Hecke operators $\mathcal{H}^{-1}$ and $\mathcal{H}$ (cf. \S\ref{S4.4}) to the moduli spaces $P_{-1}(<)$ and $P_{1}(>)$ we obtain $$P_{0}(<)\simeq P_{-1}(<) \text{\,\,\, and \,\,\,} P_{0}(>)\simeq P_{1}(>).$$ Let $\tau\in\Sigma_r$ be the cyclic permutation
$\tau\cdot(c_1,...,c_r)=(c_2,...,c_r,c_1)$, and consider two points in $V^*$: 
\begin{multline*}
	  \theta_1[k] = 
\frac{k+r}{r}\cdot(1,1,\dots,1)-(k+r)x_r-\rho= \tau\cdot\left(\frac{k}{r}-k,\frac{k}{r},...,\frac{k}{r}\right) -\tau\cdot\rho, \\
	\theta_{-1}[k] = 
-\frac{k+r}{r}\cdot(1,1,\dots,1)+(k+r)x_1-\rho =  \tau^{-1}\cdot\left(-\frac{k}{r},...,-\frac{k}{r},-\frac{k}{r}+k\right) -\tau^{-1}\cdot\rho.
\end{multline*} 
We have shown that the two polynomials 
$$\chi_-(k; \lambda) = \chi(P_0(<), \mathcal{L}(k;\lambda)) \text{\,\,  and  \, \,} \chi_+(k; \lambda) = \chi(P_0(>), \mathcal{L}(k;\lambda))$$ satisfy the following properties.
\begin{proposition}\label{linesymm}\cite[Proposition 8.5]{OTASz} 
The polynomials $$\chi_-(k; \lambda+\theta_{-1}[k]) \text{\,\,\, and \,\,\,} \chi_+(k; \lambda+\theta_{1}[k])$$ are anti-invariant under the action of the group of permutations of $\lambda_1,...,\lambda_r$.
\end{proposition}
Similarly, we define two polynomials $$\chi^\nu_<(k; \lambda) = \chi(P_0(<), \mathcal{L}(k;\lambda)\otimes\pi_!({U_\nu}\otimes\mathcal{K}^{\frac{1}{2}})),$$ $$\chi^\nu_>(k; \lambda) = \chi(P_0(>), \mathcal{L}(k;\lambda)\otimes\pi_!({U_\nu}\otimes\mathcal{K}^{\frac{1}{2}}))$$
and establish the Weyl antisymmetry for the modified polynomials 
\begin{equation*}\label{f+-}
f^\nu_<(k; \lambda) = \chi^\nu_<(k; \lambda)  - 
 \sum_\mu m_\mu\, \mu_1 \chi(P_0(<), \mathcal{L}(k;\lambda+(\mu_1,...,\mu_{r-1},\mu_r- |\nu|)))
\end{equation*}
and 
\begin{equation*}
f^\nu_>(k; \lambda) = \chi^\nu_>(k; \lambda)  + 
 \sum_\mu m_\mu \,\mu_r \chi(P_0(>), \mathcal{L}(k;\lambda+(\mu_1,...,\mu_{r-1},\mu_r-|\nu|))),
\end{equation*}
where we sum over all weights  $\mu$  of the irreducible representation $\rho_\nu$ and $|\nu|=\sum_i\nu_i$ (cf. notation on page \pageref{weightmult}).
\begin{example}
In case of rank-3 parabolic bundles and $\nu=(1,0,0)$ (cf. Example \ref{ex:2form}), we have  
$$f^\nu_<(k,\lambda)=\chi(P_0(<), \mathcal{L}(k;\lambda)\otimes\pi_!({U}\otimes\mathcal{K}^{\frac{1}{2}}))-\chi(P_0(<), \mathcal{L}(k;(\lambda_1+1,\lambda_2,\lambda_3-1));$$
$$f^\nu_>(k,\lambda)=\chi(P_0(>), \mathcal{L}(k;\lambda)\otimes\pi_!({U}\otimes\mathcal{K}^{\frac{1}{2}}))+\chi(P_0(>), \mathcal{L}(k;(\lambda_1,\lambda_2,\lambda_3)).$$
\end{example}
\begin{proposition}\label{geometricsymm}
Let $v_{\mathrm{det}}=\frac{\sum\nu_i}{r}(1,...,1,1-r);$ then the polynomials $$f^\nu_<(k; \lambda+\theta_{-1}[k]-v_{\mathrm{det}}) \text{\,\,\, and \,\,\,} f^\nu_>(k; \lambda+\theta_{1}[k]-v_{\mathrm{det}}))$$ are anti-invariant under the action of the group of permutations of $\lambda_1,...,\lambda_r$.
\end{proposition}
\begin{proof}
First, we will show the anti-invariance of the Euler characteristics of vector bundles on the moduli spaces of degree $\pm1$ parabolic bundles $P_{1}(>)$ and $P_{-1}(<)$, as it is simpler.  
Let $U[1]$ and $U[-1]$ be the universal bundles on $P_{1}(>)\times C$ and $P_{-1}(<)\times C$ that correspond to the normalized (cf. \S\ref{S2.11}) universal bundles on $P_0(>)$ and $P_0(<)$, respectively, and let
  $$\tilde{\theta}_{-1}=-\tau\cdot v_{\mathrm{det}}-\rho \text{\,\,\,\,\, and \,\,\,\,\,} \tilde{\theta}_{1}=-\tau^{-1}\cdot v_{\mathrm{det}}-\rho.$$
 Applying Serre duality for family of curves to the associated vector bundles $U[\pm1]_\nu$  (cf. \eqref{constructionUnu}) on the moduli spaces $P_{-1}(<)$ and $P_{1}(>)$, we obtain the following.
 \begin{lemma}\label{Serre-1} The Euler characteristics 
$\chi(P_{-1}(<), \mathcal{L}_{-1}(k;\lambda+\tilde{\theta}_{-1})\otimes\pi_!(U[-1]_\nu\otimes{\mathcal{K}}^{\frac{1}{2}}))$ and $\chi(P_{1}(>), \mathcal{L}_{-1}(k;\lambda+\tilde{\theta}_{1})\otimes\pi_!({U[1]_\nu}\otimes{\mathcal{K}}^{\frac{1}{2}}))$  are  anti-invariant under the permutations of $\lambda_1,...,\lambda_r$.
\end{lemma}
\begin{proof} Note that  $U[-1]\simeq p^*(UN^-)\otimes(\mathcal{F}_{2}/\mathcal{F}_1)^*$ (cf. Remark \ref{normalU}), hence $$U[-1]_\nu\simeq p^*(UN^-_\nu)\otimes(\mathcal{F}_{2}/\mathcal{F}_1)^{-\sum\nu_i},$$ where $UN^-_\nu$ is a vector bundle on $N_{-1}\times C$ obtained by \eqref{constructionUnu} from the universal bundle $UN^-$. 
Then  $$\pi_!(U[-l]_\nu\otimes{\mathcal{K}}^{\frac{1}{2}})\simeq \pi_!(p^*(UN^-_\nu)\otimes{\mathcal{K}}^{\frac{1}{2}})\otimes\mathcal{L}_{-1}(1;(0,...,0,-1,0))^{\sum\nu_i}$$ by Lemma \ref{O(1)},
and thus 
\begin{multline*}
\chi(P_{-1}(<), \mathcal{L}_{-1}(k;\lambda+\tilde{\theta}_{-1})\otimes\pi_!({U[-1]_\nu}\otimes{\mathcal{K}}^{\frac{1}{2}}) = \\
\chi(P_{-1}(<), \mathcal{L}_{-1}(k+\sum\nu_i;\lambda-\frac{\sum\nu_i}{r}(1,...,1)-\rho)\otimes\pi_!(p^*(UN^-_\nu)\otimes{\mathcal{K}}^{\frac{1}{2}})).
\end{multline*}
Since the line bundle $\mathcal{L}_{-1}(r;(-1,...,-1))$ is a pullback of the ample generator of $Pic(N_{-1})$ \cite[Lemma 8.4]{OTASz}, the statement follows from Serre duality for families of curves \cite[Proposition 8.1]{OTASz}. 
The proof for the Euler characteristic on the moduli space $P_1(>)$ is similar.
\end{proof}
Recall that our goal is to show certain antisymmetries for the polynomials $f^\nu_\lessgtr(k;\lambda)$, which are  the linear combinations of the Euler characteristics of vector bundles on the moduli spaces $P_{0}(\lessgtr)$.  We will follow the argument for the polynomial $f^\nu_<$ (the proof for  $f^\nu_>$ is analogous). 

Under the isomorphism $\mathcal{H}: P_{0}(<)\xrightarrow{\sim} P_{-1}(<),$ 
 vector bundles 
on $P_{0}(<)$  correspond to vector bundles on $P_{-1}(<)$. Below, we will write this correspondence explicitly and then will apply Lemma \ref{Serre-1} to the vector bundles on $P_{-1}(<)$ to obtain antisymmetries for the Euler characteristics.
 
Note that trivially
\begin{equation}\label{triv}
-v_{\mathrm{det}}+(\mu_1,...,\mu_{r-1},\mu_r-\sum\nu_i)= -\frac{\sum\nu_i}{r}(1,...,1)+\mu,
\end{equation}
and thus it follows from Proposition \ref{HeckeChern}, that 
\begin{multline}\label{rasp}
\chi(P_0(<), \mathcal{L}(k;\lambda+\theta_{-1}[k]-v_{\mathrm{det}})\otimes\pi_!({U_\nu}\otimes\mathcal{K}^{\frac{1}{2}}))= \\ \chi(P_{-1}(<), \mathcal{L}_{-1}(k;\tau\cdot\lambda-\frac{k}{r}(1,...,1)-\tilde{\theta}_{-1})\otimes\pi_!({U_\nu}[-1]\otimes{\mathcal{K}}^{\frac{1}{2}}))+ \\ \sum_\mu m_\mu\, \mu_1\chi(P_0(>), \mathcal{L}(k;\lambda+\theta_{-1}[k]-\frac{\sum\nu_i}{r}(1,...,1)+\mu)), 
\end{multline}

Using Lemma \ref{Serre-1} and equations \eqref{triv} and \eqref{rasp}, for any permutation $\sigma\in\Sigma_r$ we obtain
\begin{multline}
f^\nu_<(k; \sigma\cdot\lambda+\theta_{-1}[k]-v_{\mathrm{det}})\overset{\eqref{triv} \eqref{rasp}}= \\
\chi(P_{-1}(<), \mathcal{L}_{-1}(k;\tau\cdot\sigma\cdot\lambda-\frac{k}{r}(1,...,1)+\tilde{\theta}_{-1})\otimes\pi_!({U_\nu}[-1]\otimes{\mathcal{K}}^{\frac{1}{2}}))+ \\ \sum_\mu m_\mu \, \mu_1\chi(P_0(<), \mathcal{L}(k;\sigma\cdot\lambda+\theta_{-1}[k]-\frac{\sum\nu_i}{r}(1,...,1)+\mu)) - 
\\
 \sum_\mu m_\mu\, {\mu_1}\chi(P_0(<), \mathcal{L}(k;\sigma\cdot\lambda+\theta_{-1}[k]-\frac{\sum\nu_i}{r}(1,...,1)+\mu)) \overset{\ref{Serre-1}}= 
 \end{multline}
\begin{multline*}
= (-1)^\sigma\chi(P_{-1}(<), \mathcal{L}_{-1}(k;\tau\cdot\lambda-\frac{k}{r}(1,...,1)+\tilde{\theta}_{-1})\otimes\pi_!({U_\nu}[-1]\otimes{\mathcal{K}}^{\frac{1}{2}})) \overset{\eqref{rasp} }= \\
(-1)^\sigma\chi(P_{0}(<), \mathcal{L}_{0}(k;\lambda+\theta_{-1}[k]-v_{\mathrm{det}})\otimes\pi_!({U_\nu}[-1]\otimes{\mathcal{K}}^{\frac{1}{2}})) - \\
(-1)^\sigma\sum_\mu m_\mu\, {\mu_1}\chi(P_0(>), \mathcal{L}(k;\lambda+\theta_{-1}[k]-\frac{\sum\nu_i}{r}(1,...,1)+\mu)) \overset{\mathrm{def}}=\\
(-1)^\sigma f^\nu_<(k; \lambda+\theta_{-1}[k]-v_{\mathrm{det}}),
\end{multline*}
which completes the proof of Proposition \ref{geometricsymm} for $f^\nu_<$.
The proof for  $f^\nu_>$ is similar.
\end{proof}

\subsection{The Affine Weyl group}\label{S5.2}
We define an action of the \textit{affine 
Weyl group} $ \Sigma\rtimes \Lambda $ on $ \Lambda\times\Z_{>0} $, 
which acts trivially on the second factor, the level, 
and the action at level $ k>0 $ is given by  
$$ \sigma.\lambda=\sigma\cdot(\lambda+\rho+v_{\mathrm{det}})-\rho -v_{\mathrm{det}}$$ and  $$ \gamma.\lambda=\lambda+(k+r)\gamma\text{ 
\,\,\,for \,\,\,}
\sigma\in\Sigma,\,\gamma\in\Lambda. $$ We denote the 
resulting group of affine-linear transformations of $ V^* $ 
by $ \affweyl $. It is easy to verify that the stabilizer subgroup
\[ 
\Sigma_r^+\overset{\mathrm{def}}=\mathrm{Stab}(\theta_{1}[k]-v_{\mathrm{det}},
\affweyl)\subset\affweyl \]
is generated by the   
transpositions $s_{i,i+1},\;1\leq i\leq r-2$ and the reflection $\alpha^{r-1,r}\circ s_{r-1,r};$ 
similarly,
$$ 
\Sigma_r^-\overset{\mathrm{def}}=\mathrm{Stab}(\theta_{-1}[k]-v_{\mathrm{det}},
\affweyl)\subset\affweyl $$ is generated by $s_{i,i+1}$, $2\leq 
i\leq r-1$ and the reflection $ 
\alpha^{1,2}\circ s_{1,2} $.

Then Proposition \ref{geometricsymm} maybe recast in the 
following form: the polynomial $f^\nu_>(k;\lambda)$ is 
anti-invariant with respect to the copy $ 
\Sigma_r^+ $ of the symmetric group $ \Sigma_r $, while 
$f^\nu_{-}(k;\lambda)$ is anti-invariant with respect to the 
copy $ \Sigma_r^- $ of the symmetric group $ \Sigma_r $.

The following statement is straightforward:
\begin{lemma}\label{lem:generate}
Both subgroups $ \Sigma_r^\pm$ are isomorphic to $\Sigma_r$ and for $r>2$ the two subgroups generate the affine Weyl 
group $ \affweyl $.
\end{lemma}

\subsection{Symmetries in residue formulas}\label{S5.3}
The main result of this section is Proposition \ref{propsymmIBer}, where we show  the antisymmetries for the residues formulas on the right-hand side of \eqref{eqmain}.

Recall that in \S\ref{S5.1} we defined a pair of polynomials $\chi^\nu_\lessgtr$  corresponding to the Euler characteristics from the left-hand side of  \eqref{eqmain} and proved the Weyl antisymmetry for the modified polynomials $f^\nu_\lessgtr$. Now we
  define the two polynomials  corresponding to the residue expressions from the right-hand side of  \eqref{eqmain}: \label{F+-}
$$R^\nu_>(k;\lambda)=N_r\cdot \frac{\partial }{\partial \delta}\big{|}_{\delta=0} \sum_{\bb\in\DD}
	\iber_{\bb, Q}\left[\mathrm{Hess}(Q(x))^{g-1}w_\Phi^{1-2g}(x)\exp(\scp{\lala+v_{\mathrm{det}}}{x})\right] \left( -[{\theta_1}]_{\bb}     \right)$$ and
$$R^\nu_<(k;\lambda)=N_r\cdot \frac{\partial }{\partial \delta}\big{|}_{\delta=0} \sum_{\bb\in\DD}
	\iber_{\bb, Q}\left[\mathrm{Hess}(Q(x))^{g-1}w_\Phi^{1-2g}(x)\exp(\scp{\lala+v_{\mathrm{det}}}{x})\right] \left( -[{\theta_{-1}}]_{\bb}     \right),$$
	where $ \theta_1=\frac1r\cdot(1,1,\dots,1)-x_r $, and
$ \theta_{-1}=-\frac1r\cdot(1,1,\dots,1)+x_1 $,	
and  establish the Weyl antisymmetry for the modified pair of polynomials:
\begin{multline*}
F^\nu_>(k;\lambda)= R^\nu_>(k;\lambda) + N_{r,k}\cdot \\ 
	\sum_\mu m_\mu\, \mu_r\sum_{\bb\in\DD} \iber_{\bb, (k+r)K}\left[w_\Phi^{1-2g}(x)\exp(\scp{\lala+(\mu_1,...,\mu_{r-1},\mu_r-|\nu|)}{x})\right]\left( -[{\theta_1}]_{\bb}     \right)
\end{multline*}
and 
\begin{multline*}
F^\nu_<(k;\lambda)= R^\nu_<(k;\lambda) -  N_{r,k}\cdot  \\ 
	\sum_\mu m_\mu\, \mu_1\sum_{\bb\in\DD} \iber_{\bb, (k+r)K}\left[w_\Phi^{1-2g}(x)\exp(\scp{\lala+(\mu_1,...,\mu_{r-1},\mu_r-|\nu|)}{x})\right]\left( -[{\theta_{-1}}]_{\bb}     \right),
\end{multline*}
where, as usual,  the sum runs over all weights $\mu$  of the irreducible representation 
$\rho_\nu$ and $|\nu|=\sum_i\nu_i$ (cf. notation on page \pageref{weightmult}).

\begin{proposition}\label{propsymmIBer}
The polynomial $ F^\nu_>(k;\lambda) $ is anti-invariant with 
respect to $ \Sigma_r^+ $, and $ F^\nu_<(k;\lambda) $ is 
anti-invariant with respect to $ \Sigma_r^- $.
\end{proposition}
\begin{proof}
We first consider a generator of $ \Sigma^-$ of the type   $ \sigma=s_{i,i+1} $, $ 2\leq i\leq r-1 $.  Note that 
$$\sigma.\lambda+\rho+v_{\mathrm{det}}=\sigma(\lambda+\rho+v_{\mathrm{det}}) \text{\,\,\,\, and \,\,\,\,} \sigma.\lambda+\rho+\mu-|\nu|x_r = \sigma(\lambda+\rho-|\nu|x_r)+\mu.$$
Using \cite[Lemma 4.5]{OTASz} and the facts that $$\sigma\cdot\mathrm{Hess}(Q(x))=\mathrm{Hess}(Q(x)) \text{\,\,\,\,  and \, \,\,\,} \sigma\cdot w_\Phi^{1-2g}(x)= -w_\Phi^{1-2g}(x),$$ we obtain
\begin{multline*}
F^\nu_<(k;\sigma.\lambda)=  \\ N_r \frac{\partial }{\partial \delta}\big{|}_{\delta=0}   \sum_{\bb\in\DD}
	\iber_{\bb, Q}[-\mathrm{Hess}(Q(x))^{g-1}w_\Phi^{1-2g}(x)   \exp\scp{(\lambda+\rho+v_{\mathrm{det}}}{x}]  \left( -\sigma^{-1}\cdot[{\theta_{-1}}]_{\bb}     \right)  - \\  N_{r,k}
	\sum_\mu m_\mu\, \mu_1  \sum_{\bb\in\DD} \iber_{\bb, (k+r)K}\left[-w_\Phi^{1-2g}(x)\exp(\lala+\sigma^{-1}\cdot\mu-|\nu|x_r))\right]\left( -\sigma^{-1}\cdot[{\theta_{-1}}]_{\bb}    \right) =  \\  -F^\nu_<(k;\lambda).
\end{multline*}
For the last equality we used the Weyl-invariance of the multiplicities of weights $\mu$ of the irreducible representation $\rho_\nu$. 

The case of the last generator $\sigma= \alpha^{1,2}\circ s_{1,2}$ requires some extra observations. 
Since 
$$\sigma.\lambda+\rho+\mu-\sum\nu_ix_r =  s_{1,2}\cdot(\lambda+\rho)+\mu-|\nu|x_r+(k+r)(x_1-x_2)$$
and
\begin{equation}\label{thetashift}
s_{1,2}\cdot[{\theta_{-1}}]=[{\theta_{-1}}]-(x_1-x_2),
\end{equation}
we have
\begin{multline}\label{ressymsecond}
\sum_\mu m_\mu\, \mu_1\sum_{\bb\in\DD} \iber_{\bb, (k+r)K}\left[w_\Phi^{1-2g}(x)\exp(\sigma.\lambda+\rho+\mu-|\nu|x_r))\right]\left( -[{\theta_{-1}}]_{\bb}    \right) = \\
\sum_\mu m_\mu\, \mu_1\sum_{\bb\in\DD} \iber_{\bb, (k+r)K}\left[-w_\Phi^{1-2g}(x)\exp(\lala+s_{1,2}\cdot\mu-|\nu|x_r-(k+r)(x_1-x_2))\right] \\ \left( -s_{1,2}\cdot[{\theta_{-1}}]_{\bb}    \right) = 
-\sum_\mu m_\mu\, \mu_2\sum_{\bb\in\DD} \iber_{\bb, (k+r)K}\left[w_\Phi^{1-2g}(x)\exp(\lala+\mu-|\nu|x_r))\right]\left( -[{\theta_{-1}}]_{\bb}\right).
\end{multline}
Note that 
$$\sigma.\lambda+\rho+v_{\mathrm{det}}=s_{1,2}\cdot(\lambda+\rho)+v_{\mathrm{det}}+(k+r)(x_1-x_2),$$ 
hence, using \eqref{ressymsecond}, we obtain 
\begin{multline}\label{ressymthird}
F^\nu_<(k;\sigma.\lambda)=N_r \frac{\partial }{\partial \delta}\big{|}_{\delta=0} \sum_{\bb\in\DD}
	\iber_{\bb, Q}[-\mathrm{Hess}(Q(x))^{g-1}w_\Phi^{1-2g}(x)\\ \exp(\lambda+\rho+ 
	v_{\mathrm{det}}-(k+r)(x_1-x_2))] \left( -s_{1,2}\cdot[{\theta_{-1}}]_{\bb}     \right)+ \\ N_{r,k} 
	\sum_\mu m_\mu\, \mu_2\sum_{\bb\in\DD} \iber_{\bb, (k+r)K}\left[w_\Phi^{1-2g}(x)\exp(\lala+\mu-|\nu|x_r))\right]\left( -[{\theta_{-1}}]_{\bb}    \right).
\end{multline}
Now using \eqref{thetashift} and applying Lemma \ref{trivialshift} with $w=x_1-x_2$ to \eqref{ressymthird}, we calculate that 
\begin{multline*}
F^\nu_<(k;\sigma.\lambda)=N_r \frac{\partial }{\partial \delta}\big{|}_{\delta=0} \sum_{\bb\in\DD}
	\iber_{\bb, Q}[-\mathrm{Hess}(Q(x))^{g-1}w_\Phi^{1-2g}(x)  \exp(\lala+ 
	v_{\mathrm{det}})] \left( -[{\theta_{-1}}]_{\bb}     \right)- \\ N_{r,k}\cdot
	  \sum_{\bb\in\DD}
	\iber_{\bb, (k+r)K}[-w_\Phi^{1-2g}(x)\phi_{\check{x}_{12}}(x) \exp(\lala+ 
	v_{\mathrm{det}})] \left( -[{\theta_{-1}}]_{\bb}     \right)+ \\
	N_{r,k}\cdot\sum_\mu m_\mu\, \mu_2\sum_{\bb\in\DD} \iber_{\bb, (k+r)K}\left[w_\Phi^{1-2g}(x)\exp(\lala+\mu-|\nu|x_r))\right]\left( -[{\theta_{-1}}]_{\bb}    \right).
\end{multline*}
Finally, applying the following trivial equality $$\phi_{\check{x}_{12}}(x) \exp\scp{v_{\mathrm{det}}}{x} = \sum_\mu m_\mu\, (\mu_1-\mu_2)\exp(\mu(x))$$ 
to the last two summands in our expression for polynomial $F^\nu_<(k;\sigma.\lambda)$, we conclude that 
$F^\nu_<(k;\sigma.\lambda)=-F^\nu_<(k;\lambda).$ This finishes the proof of the anti-invariance of the polynomial $F^\nu_<(k;\lambda)$; the proof for $F^\nu_>(k;\lambda)$ is similar.
\end{proof}

Note that the two differences $\chi^\nu_<-f^\nu_<$ and $\chi^\nu_>-f^\nu_>$ (cf. page \pageref{f+-}) have the form of a linear combination of the Euler characteristics of line bundles on the moduli spaces of parabolic bundles; while the differences $R^\nu_<-F^\nu_<$ and $R^\nu_>-F^\nu_>$ (cf. page \pageref{F+-}) may be written as an iterated residue of a meromorphic functions. Then using the residue formula for the Euler characteristic of line bundles, Theorem \ref{verlinde}, 
we arrive at the following statement.
\begin{proposition}\label{diffshifted} For polynomials $R^\nu_>,R^\nu_<$, $\chi^\nu_>, \chi^\nu_<$,  $ F^\nu_>,  F^\nu_<$ and $f^\nu_>, f^\nu_<$ defined on pages \pageref{F+-} and \pageref{f+-}, we have: 
\begin{equation*}
\begin{split}
\chi^\nu_>(k;\lambda)-f^\nu_>(k;\lambda) = R^\nu_>(k;\lambda)-F^\nu_>(k;\lambda); \\
\chi^\nu_<(k;\lambda) -f^\nu_<(k;\lambda)= R^\nu_<(k;\lambda)-F^\nu_<(k;\lambda).
\end{split}
\end{equation*}
\end{proposition}

\end{section}
\begin{section}{Proof of Theorem \ref{main} and some generelizations of our result}
In this section, we finish the proof of our main result and present some of its generalizations.
\subsection{Proof of Theorem \ref{main}}\label{S6.1}
The proof of Theorem \ref{main} follows the logic of \cite{OTASz}. In this section, we repeat the argument with only minor changes. 

 Recall that in \S\ref{S2.2} we introduced a chamber structure on $\Delta\subset 
V^*$ created by the walls $S_{\Pi,l}$, where $\Pi=(\Pi^\p,\Pi^{\pp})$ is a
 nontrivial partition, and $ l\in\Z $.  
Denote by 
$$\widecheck{\Delta} = \{(k;a)| \, a/k\in\Delta\} \subset \mathbb{R}_{>0}\times V^*$$
 the cone over $\Delta\subset V^*$, and let  
 $$\widecheck{\Delta}^{\mathrm{reg}} = \{(k;a)| \, a/k\in\Delta \, \text{is 
 regular}\}\subset \widecheck{\Delta}$$ be the set of its regular points. 
 Denote by $\widecheck{S}_{\Pi,l}\subset\widecheck{\Delta}$  the cone over the 
 wall $S_{\Pi,l}\subset \Delta$; then $\widecheck{\Delta}^{\mathrm{reg}}$ is 
 the complement of the union of walls $\widecheck{S}_{\Pi,l}$ in 
 $\widecheck{\Delta}$. Finally,  denote by  
 $\widecheck{\Delta}^{\mathrm{reg}}_\Lambda$ the intersection of the lattice 
 $\mathbb{Z}_{>0}\times\Lambda$ with $\widecheck{\Delta}^{\mathrm{reg}}$.

By substituting ${c} =\lambda/k$, we can consider the left-hand side and the right-hand side of the equation in Theorem 
\ref{main} as functions in $(k,\lambda)\in\widecheck{\Delta}^{\mathrm{reg}}_\Lambda$. We denote by $\chi(k;\lambda)$ and $R(k;\lambda)$  the left-hand side and the right-hand side, correspondingly. 

We showed that $\chi(k;\lambda)$ and $R(k;\lambda)$  are \textit{polynomials} on the cone over each chamber in $\Delta$ (cf. \S\ref{S2.2}, \S\ref{S3.1}). 
We proved that the wall-crossing terms, i.e. the differences between polynomials on neighbouring chambers, for $\chi(k;\lambda)$ (cf. Proposition \ref{wcrprop})  and for $R(k;\lambda)$ (cf. Corollary \ref{corwcresidue}) coincide, hence there exists a polynomial $\Theta(k;\lambda)$ on $\mathbb{Z}_{>0}\times\Lambda$, such that the restriction of $\Theta(k;\lambda)$ to $\widecheck{\Delta}^{\mathrm{reg}}_\Lambda$ is equal to  the difference $\chi(k;\lambda)-R(k;\lambda)$. 

Now for $r>2$, we can conclude that 
$$\Theta(k;\lambda)=\chi^\nu_>(k;\lambda)-R^\nu_>(k;\lambda)=\chi^\nu_{<}(k;\lambda)-R^\nu_{<}(k;\lambda),$$
 where $\chi^\nu_{\gtrless}(k;\lambda)$ and $R^\nu_{\gtrless}(k;\lambda)$  are the restrictions of $\chi(k;\lambda)$ and $R(k;\lambda)$  to two specific chambers defined in   \cite[Lemma 8.3]{OTASz}. 
 Then, according to Proposition \ref{diffshifted}, 
 $$\Theta(k;\lambda)= f^\nu_>(k;\lambda)-F^\nu_>(k;\lambda) = f^\nu_<(k;\lambda)-F^\nu_<(k;\lambda).$$
 It follows from Propositions \ref{geometricsymm} 
and \ref{propsymmIBer} that the polynomial $\Theta(k;\lambda)$ is 
anti-invariant with respect to the action of the subgroups $\Sigma^\pm_r$ (cf. the end of \S\ref{S5.2}), and hence by 
Lemma \ref{lem:generate}, it is anti-invariant under the action of the entire 
affine Weyl group $\affweyl$. It is easy to see that any such polynomial 
function has to vanish, and thus $\chi(k;\lambda)=R(k;\lambda)$.

As marked above, the argument does not work for $r=2$, since in this case the groups $\Sigma^+_r$ and $\Sigma^-_r$ (cf. \S\ref{S5.2}) coincide, and thus they do not generate the entire affine Weyl group. A  solution is to consider the 2-punctured case, treated in \S\ref{S1.1}-\S\ref{S1.3}; this finishes the proof of Theorem \ref{main}.

\subsection{Generalization}\label{S6.2}

 Now  we formulate a mild generalization of our result, Theorem \ref{maingen}, and explain, following an idea of Teleman and Woodward \cite{TelemanW}, how our formulas can be used to calculate the Euler characteristic of a more general class of vector bundles on the moduli spaces of parabolic vector bundles. 

Let  $\nu[1],...,\nu[m]$ be dominant weights of $GL_r$. 
Replacing $Q$ and $v_{\mathrm{det}}$  in Theorem \ref{main} by the multi-parameter version
$${\bf{Q}}=(k+r)K-\sum_j\delta_j\phi^{\nu[j]}, \,\,\,\,\, 
{\bf{v}}_{\mathrm{det}}=\sum_{j=1}^m(1,...,1,1-r)\frac{\sum_i\nu[j]_i}{r},$$
 we can deduce the following Theorem. 

\begin{theorem}\label{maingen}
Let  ${\bf{Q}}$ and ${\bf{v}}_{\mathrm{det}}$ be as above, let $\mathcal{K}$ be the  canonical class of the curve $C$,  $ \lambda \in\Lambda$, $ k\in\mathbb{Z}_{>0} $, $\nu=(\nu_1\geq\nu_2 ... \geq\nu_r)\in\mathbb{Z}^r$, $\lala=\lambda+\rho$,  and let ${c}\in\Delta$ be a regular element (cf. page \pageref{regular}).
Then for any diagonal basis $ 
	\DD \in \Bases $, the following equality holds:
\begin{multline*}
	\chi(P_0(c), \L(k;\lambda)\otimes\pi_!(U_{\nu[1]}\otimes\mathcal{K}^{\frac{1}{2}})\otimes\pi_!(U_{\nu[2]}\otimes\mathcal{K}^{\frac{1}{2}})\otimes...\otimes\pi_!(U_{\nu[m]}\otimes\mathcal{K}^{\frac{1}{2}}))
	= \\ 
	N_r\cdot \frac{\partial^m}{\partial\delta_1 . . . \partial\delta_m}\Big{|}_{\delta_1=. . . =\delta_m=0} \sum_{\bb\in\DD}
	\iber_{\bb, {\bf{Q}}}\left[\mathrm{Hess}({\bf{Q}}(x))^{g-1}w_\Phi^{1-2g}(x)\exp(\scp{\lala+{\bf{v}}_{\mathrm{det}}}{x})\right] \left( -[{c}]_{\bb}     \right).
\end{multline*}
\end{theorem}
The proof of this theorem is analogous to  our proof of Theorem \ref{main}.

Using Theorem \ref{maingen} one can also obtain formulas for the Euler characteristics of vector bundles, which involve the exterior powers  $\bigwedge^{\raisebox{-0.4ex}{\scriptsize $l$}} \pi_!(U_{\nu}\otimes\mathcal{K}^{\frac{1}{2}}))$. Let us briefly explain the case
\begin{equation}\label{wedge2}
\chi\left(P_0(c), \L(k;\lambda)\otimes \bigwedge\nolimits^2 \pi_!(U_{\nu}\otimes\mathcal{K}^{\frac{1}{2}})\right).
\end{equation}

Recall that the $n$-th Adams operator  $\psi^n$  is defined by $\psi^nL=L^n$ for a line bundle $L$ and extends to $K$-theory additively by the splitting principle. 
It follows from the Groethendieck-Riemann-Roch theorem and equation \eqref{part2} that 
\begin{multline}\label{wedgeGRR}
ch(\psi^n (\pi_!(U_\nu\otimes\mathcal{K}^{\frac{1}{2}}))) = \sum_{i\geq0}n^i\cdot ch_i(\pi_!(U_\nu\otimes\mathcal{K}^{\frac{1}{2}})) = \\ \frac{1}{n}\sum_{i\geq1}n^{i}\cdot \pi_*(ch_{i}(U_\nu)) =
\frac{1}{n}\pi_*ch(\psi^n(U_\nu))= \frac{1}{n}ch(\pi_!(\psi^n(U_\nu)\otimes\mathcal{K}^{\frac{1}{2}})).
\end{multline}
\end{section}
Since for any vector bundle $V$ 
$$ch\left(\bigwedge\nolimits^2 V\right)=\frac{ch(V^{\otimes 2})-ch(\psi^2V)}{2},$$ the Euler characteristic \eqref{wedge2} equals 
\begin{multline*}
 \frac{1}{2}\chi(P_0(c), \L(k;\lambda)\otimes (\pi_!(U_{\nu}\otimes\mathcal{K}^{\frac{1}{2}}))^2) - \frac{1}{4}\chi(P_0(c), \L(k;\lambda)\otimes \pi_!(\psi^2(U_{\nu})\otimes\mathcal{K}^{\frac{1}{2}})).
\end{multline*}
Finally,  note that the character function (cf. page \pageref{diagram}) for $\psi^n(U_\nu)$ is  $\phi^\nu(x^n)$, hence using Theorem \ref{maingen}, we obtain the formula for the Euler characteristic \eqref{wedge2}.

\end{document}